\documentclass[twoside]{amsart}

\numberwithin{equation}{section}
\usepackage{amsmath}
\usepackage{amssymb,amsfonts,amsthm,latexsym,psfrag}
\usepackage[mathscr]{euscript}
\usepackage{graphicx,color,epic,eepic,bbold,enumerate}
\usepackage[all]{xy}
\usepackage{hyperref}

\theoremstyle{plain}
        \newtheorem{theorem}{Theorem}[section]
        \newtheorem{lemma}[theorem]{Lemma}
        \newtheorem{proposition}[theorem]{Proposition}
        \newtheorem{corollary}[theorem]{Corollary}
               
\theoremstyle{definition}
        \newtheorem{definition}[theorem]{Definition}
        \newtheorem{remark}[theorem]{Remark}
        \newtheorem{example}[theorem]{Example}

\newcommand{\R}{\mathbb{R}}

\newcommand{\N}{\mathbb{N}}

\newcommand{\calC}{\mathcal{C}}

\newcommand{\calO}{\mathcal{O}}

\newcommand{\calE}{\mathcal{E}}

\newcommand{\calJ}{\mathcal{J}}
\newcommand{\calS}{\mathcal{S}}

\newcommand{\calZ}{\mathcal{Z}}

\newcommand{\sfE}{\mathsf{E}}
\newcommand{\sfJ}{\mathsf{J}}
\newcommand{\sfR}{\mathsf{R}}
\newcommand{\sfT}{\mathsf{T}}
\newcommand{\sfW}{\mathsf{W}}
\newcommand{\frg}{\mathfrak{g}}
\newcommand {\cc} [1] {\overline {{#1}}}
\newcommand {\id} {\operatorname{id}}
\newcommand {\Kern} {\operatorname{ker}}

\newcommand{\bs}{\boldsymbol}

\begin{document}
\title{Invariant Whitney Functions}

\author{Hans-Christian~Herbig and Markus Pflaum}
\address{Hans-Christian~Herbig\newline\indent
        Universidade Federal de Rio de Janeiro,
        Departamento de Matem\'atica Aplicada\newline\indent
        CEP 21941-909 - Rio de Janeiro, Brazil}
\email{herbighc@gmail.com} 
\address{Markus J. Pflaum\newline\indent
         Department of Mathematics, University of Colorado, 
         Boulder, CO 80309-0395, USA}   
\email{markus.pflaum@colorado.edu}

\begin{abstract} 
A theorem of Gerald Schwarz \cite[Thm.~1]{SchwaSFIACLG} says that for a linear action of a compact Lie group $G$ on a finite dimensional real vector space $V$ any smooth $G$-invariant function on $V$ can be written as a composite with the Hilbert map. We prove a similar statement for the case of Whitney functions along a subanalytic set $Z\subset V$ fulfilling some regularity assumptions. In order to deal with the case when $Z$ is not $G$-stable we use the language of groupoids.
\end{abstract}
\maketitle
\tableofcontents
\bibliographystyle{amsplain}
\section{Introduction} \label{sec:intro}

The purpose of this paper  is to present an analogue of the Theorem of Gerald Schwarz \cite[Theorem 1]{SchwaSFIACLG} on differentiable invariants of representations of a compact group for the case of invariant Whitney functions. The main difficulty to overcome is to include 
also the case of Whitney functions along subsets that are not stable under the action of the group. We develop the theory as a tool for Hochschild homology calculations for algebras of smooth functions on orbit spaces \cite{HerPflHHASFOS}. We hope that the results are of independent interest.

To explain our findings we start by recalling Schwarz's result. 
Suppose that $G$ is a compact Lie group and $V$ is a finite dimensional vector space over the field of real numbers $\mathbb R$ on which $G$ acts linearly. Without loss of generality we can assume that the action of $G$ preserves a euclidean scalar product $\langle\;,\;\rangle$, which means that we actually have an 
orthogonal representation $G\to \operatorname O(V)$ into the orthogonal group of euclidean vector space $(V,\langle\;,\;\rangle)$. For the  \emph{orbit space} of the $G$-action on $V$ we write $V/G$. 

By the Theorem of Hilbert and Weyl (see for example \cite[\S 98] {ZelCLGR}) there exists a complete system of polynomial invariants $\rho_1,\dots ,\rho_\ell\in \R[V]^G$, which we choose to be homogeneous and minimal. We refer to $\rho_1,\dots ,\rho_\ell$ as the \emph{Hilbert basis} of the representation. By the \emph{Hilbert map} we mean the vector valued regular map $\rho=(\rho_1,\dots ,\rho_\ell): V\to \mathbb R^\ell$. 
The Hilbert map  descends to a proper embedding $\cc \rho: V/G\to \R^\ell$, i.e.~the diagram 
\begin{eqnarray*}
\xymatrix{V\ar[r]^{\rho\qquad\quad}\ar[d]_{\pi}  & \:\: 
  X:=\rho(V)\subset \mathbb R^\ell&\\
  V/G\ar[ru]_{\cc\rho}&&}
\end{eqnarray*}
commutes. Here $\pi$ denotes the orbit map. We refer to $\cc\rho$ as the \emph{Hilbert embedding}.
By the Tarski-Seidenberg principle, the image 
$X=\rho(V)\subset \mathbb R^\ell$ of the polynomial map $\rho$ is a semialgebraic subset. A more constructive approach to describing the semialgebraic set $X$ has been elaborated in \cite{ProSchIDOS}.

\begin{theorem} \label{SchwaSFIACLG} Let $G\to \operatorname{O}(V)$ be a finite dimensional orthogonal representation of the compact Lie group $G$. Let $\rho_1,\dots,\rho_\ell\in\mathbb R[V]^G$ be a minimal complete system of polynomial invariants and $\rho:=(\rho_1,\dots,\rho_\ell): V\to \mathbb R^\ell$ the corresponding Hilbert map. Then the pullback 
\[\rho^*:\calC^\infty(\mathbb R^\ell)\to\calC^\infty(V)^G\] 
is split sujective. That means that there exists a continuous map $\lambda:\calC^\infty(V)^G\to\calC^\infty(\mathbb R^\ell)$ such that $\rho^*\circ \lambda$ is the identity map $\calC^\infty(V)^G\to\calC^\infty(V)^G$.
\end{theorem}

The surjectivity of $\rho^*$ has been established by Gerald Schwarz \cite[Thm.~1]{SchwaSFIACLG}. The existence of the continuous split has been proven later by 
John Mather in \cite{MatDI}. For an exposition of the material the reader might want to consult the monograph \cite{BierstonesBook}. 

It is natural to regard the orbit space $V/G$ as a topological space with a \emph{smooth structure} $\mathcal C^\infty(V/G):=\mathcal C^\infty(V)^G$. Similarly $X:=\rho(V)$ carries the smooth structure $\mathcal C^\infty(X)=\{f\in \calC(X)\mid \exists F\in \calC^\infty(\mathbb R^\ell):\: F_{|X}=f\}$. As a corollary to Theorem  \ref{SchwaSFIACLG} the Hilbert embedding $\cc\rho$ is a \emph{diffeomorphism} of $V/G$ onto $X$, i.e., a homeomorphism such that 
$\cc\rho^*:\mathcal C^\infty(X)\to \mathcal C^\infty(V/G)$ is an isomorphism of algebras.
 
We refer to Theorem \ref{SchwaSFIACLG} as the Theorem of Schwarz and Mather on differentiable invariants. If we do not refer to the continuous split we speak of the Theorem of Schwarz. 

Our objective is to prove an analogue of the Theorem of Schwarz for  Whitney fields $\calE^\infty(Z)$ along a 
locally closed semialgebraic or subananlytic set $Z\subset V=\mathbb R^n$. Note that we do not assume $Z$ to be 
$G$-stable. In order to speak of invariance of such Whitney fields we make use of the language of groupoids. \\

{\bf Acknowledgements.}
The authors express their gratitude to Edward Bierstone for clarifications 
concerning  Gabrielov regularity. They thank the Instituto de Matem\'atica 
Pura e Aplicada (IMPA) in Rio de Janeiro and the Center for the Quantum
Geometry of Moduli Spaces (QGM) at Aarhus University for hospitality.
The research of MP has been supported by the National Science Foundation 
under DMS-1105670 and by Simons Foundation collaboration grant number 359389.
The research of HCH has received support from QGM which is funded by the Danish 
National Research Foundation. Moreover, he received support from the
Brazilian National Council for Scientific and Technological Development CNPq,
the Austrian Ministry of Science and Research BMWF, Start-Prize Y377, and from
the Eduard {\v C}ech Institute through the grant GA CR P201/12/G028.
HCH also gratefully acknowledges hospitality by the University of Colorado. 
\section{Preliminaries on Whitney functions}
\label{sec:whitney-functions}
For the convenience of the reader and to set up notation we review in
this section the basic concepts of the theory of Whitney functions. 
We also provide the proof of a folklore theorem on the pullback of Whitney functions, 
which we could not find in the literature.
For further information on Whitney functions see \cite{MalIDF}.

\begin{definition}
  Let $X\subset \R^n$ be a locally closed subset, which means that there is
  an open $U\subset \R^n$ such that  $X\subset U$ is relatively  closed. 
  Assume that $m \in \N\cup \{ \infty\}$. Then denote by
  $\sfJ^m (X)$ the space of all \textit{jets of order $m$} 
  (or \textit{$m$-jets}) on $X$ which means the vector space of all families
  \[
    F = \big( F_\alpha \big)_{\alpha \in \N^n, \, |\alpha| \leq m}
    \quad \text{with} \quad F_\alpha \in \calC (X) . 
  \]
  For $K\subset X$ compact and every natural $k \leq m$ we define the 
  seminorm
  $| \cdot |_{K,k}$ on $\sfJ^m (X)$ by
  \[
    | F|_{K,k} := \sup_{x\in K\atop |\alpha| \leq k} | F_\alpha (x) | \: .
  \]
  For $\beta \in \N^n$ with $| \beta | \leq m$ denote by 
  $\partial^\beta : \sfJ^m (X) \rightarrow \sfJ^{m-|\beta|} (X)$ the 
  linear map given by 
  \[
    \big( F_\alpha \big)_{\alpha \in \N^n, \, |\alpha| \leq m}
    \mapsto \big( F_{\alpha+\beta} \big)_{\alpha \in \N^n, \, 
    |\alpha| \leq m-|\beta|} \: .
  \]

  Given two jets $E,F\in \sfJ^m (X)$ one defines their product 
  $ E F\in \sfJ^m (X)$
  as the jet with components 
  \[
     ( EF )_\alpha := \sum_{\beta + \gamma = \alpha \atop \beta , \gamma \in \N^n}
     E_\beta F_\gamma, \quad |\alpha| \leq m.
  \]
 
  Finally, by the symbol $\sfJ^m $ we will denote the map 
  \[
   \sfJ^m=\sfJ^m_X: \calC^m (U) \rightarrow \sfJ^m (X), \:
   f \mapsto \big( (\partial^\alpha f)_{|X} \big)_{\alpha \in \N^n, \, |\alpha| \leq m} . 
  \]
\end{definition}
  The space $\sfJ^m (X)$ together with the topology generated by the family of 
  seminorms $| \cdot |_{K,k}$ as above forms a Fr\'echet space respectively 
  a Banach space in case $X$ is compact and $m$ finite. Moreover, the maps 
  $\partial^\alpha$ and $\sfJ^m$ all become continuous linear maps with respect to 
  this topology. Finally, the product of jets is associative, and $\sfJ^m (X)$
  becomes a Fr\'echet resp.~Banach algebra. 

  For later purposes let us briefly describe at this point the action of vector fields on 
  jets. Let $F \in \sfJ^\infty (X)$, and $\xi $ a smooth vector field defined 
  on an open neighborhood of $X$. Represent the vector field as $\xi = \sum_{i=1}^n \xi_i \partial_i$,
  where the coefficients $\xi_i$ are   uniquely 
  determined smooth functions on the domain of the vector field. Then, one puts
\[
   \xi F := \sum_{i=1}^n \sfJ^\infty (\xi_i) \, \partial_iF \: . 
\]
One checks immediately, that $\xi$ acts as a derivation on $\sfJ^\infty (X)$.

  It is the goal of the following considerations to provide an
  explicit representation of the image of the map 
  $\sfJ^m : \calC^m (U) \rightarrow \sfJ^m (X)$. To this end let us first define for 
  $F \in \sfJ^m (X)$, $a\in X$ and natural $k\leq m$ a function 
  $\sfT_a^k F \in \calC^\infty (\R^n)$ by
  \[
     \sfT_a^k F (x) := \sum_{\alpha \in \N^n \atop | \alpha | \leq k}
     \frac{(x-a)^\alpha}{\alpha !} F_\alpha (a), \quad \text{for $x\in \R^n$}.
  \]
  Then put
  \[
    \widetilde{\sfT}_a^k F := \sfJ^k (\sfT_a^k F) \: \text{ and } \:
    \sfR_a^k F := F -  \widetilde{\sfT}_a^k F .
  \]
  \begin{definition}
    Let $X\subset \R^n$ be a locally closed subset. An element $F \in \sfJ^m (X)$
    with $m\in \N$ is called a \textit{Whitney function of order $m$} 
    on $X$, if for every compact set $K\subset X$ and every $\alpha \in \N^n$ 
    with $| \alpha | \leq m$ the following relation holds true:
    \[
      \left( \sfR^m_x \right)_\alpha (y) = o ( | x-y|^{m-|\alpha|}) \quad 
      \text{for $x,y\in K$ as $|x-y| \rightarrow 0$}. 
    \]
    An element $F \in \sfJ^\infty (X)$
    is called a \textit{Whitney function of order $\infty$}, if for every $m\in \N$
    the $m$-jet $\sfJ^m F$ is a Whitney function of order $m$.
    For every $m\in \N \cup \{ \infty \} $ we will denote by $\calE^m (X) $ 
    the space of Whitney functions of order $m$ on $X$.
  \end{definition}
 
For  $m\in \N \cup \{ \infty \}$,  there is system of seminorms $\|\cdot \|_{K,k}$ on the space  
$\calE^m (X)$ indexed by compact subsets  $K\subset X$ and integers $k\le m$. This system is
defined as follows. For each $k\in\N$ such that $k\le m$ and each compact $K\subset X$ put
\begin{align*}
\| F \|'_{K,k}:=\sup_{x,y\in K, x\ne y\atop \alpha\in \N^n,|\alpha|\le k}{\left|(\sfR^k_x F)_\alpha(y)\right|\over|x-y|^{k-|\alpha|}}
\end{align*} 
and set $\| F \|_{K,k}:= |F |_{K,k}+\| F \|'_{K,k}$. The space of Whitney functions  $\calE^m (X)$ together 
with the system of seminorms $(\| \cdot \|_{K,k})_{K,k}$ forms a  Fr\'echet algebra. In the case when 
$X\subset \R$ is Whitney regular, the system of seminorms 
$(\| \cdot \|_{K,k})_{K,k}$ is equivalent to the system of seminorms $(| \cdot |_{K,k})_{K,k}$. For more 
details we refer the reader to \cite[Prop.~2.6 \& Prop.~3.11]{TouIFD}.

The following fundamental theorem by Whitney 
determines in particular the image of the map $\sfJ^m$. 

\begin{theorem}[Whitney's Extension Theorem \cite{WhiAEDFCS}, 
 cf.~also {\cite[Sec.~I]{MalIDF}}]\label{WhitneyExtension}
 Let $X \subset \R^n$ be locally closed, and $U\subset \R^n$ be open. 
 Assume that $m\in \N$. Then there exists a continuous linear section
 \[
  \sfW^m : \calE^m (X) \rightarrow \calC^m (U) 
 \] 
 of the jet map $\sfJ^m : \calC^m (U) \rightarrow \sfJ^m (X)$, which in other words
 means that $\sfJ^m \circ \sfW^m = \id_{\calE^m (X)}$. In particular this implies
 that $\calE^m (U) \cong \calC^m (U)$.

 Assume now that $m \in \N \cup \{ \infty \}$ and that $Y\subset X$ is closed. 
 Under this assumption denote by $\calJ^m (Y,X)$ the kernel of the restriction 
 map 
 \[
      \operatorname{res}^X_Y : \calE^m (X) \rightarrow \calE^m (Y), \: \big( F_\alpha \big)_{\alpha \in \N^n, \, |\alpha| \leq m}
      \mapsto \big( (F_\alpha)_{|Y} \big)_{\alpha \in \N^n, \, |\alpha| \leq m} \: . 
 \]
 As a consequence, the sequence 
 \[
  0 \xrightarrow{\hspace{1.5em}} \calJ^m(Y,X) 
    \xrightarrow{\hspace{1.5em}} \calE^m (X) 
    \overset{\operatorname{res}^X_Y}{\xrightarrow{\hspace{1.5em}}} 
    \calE^m (Y) \xrightarrow{\hspace{1.5em}} 0
 \]
 is short exact.
\end{theorem}

It is an immediate consequence of Whitney's Extension Theorem that the action of 
$\partial^\alpha$ and more generally of smooth vector fields on the jet space $J^\infty (X)$ leaves the 
subspace $\calE^\infty (X)$ of Whitney functions invariant.
Likewise, the product of two Whitney functions of order $m \in \N^*\cup \{ \infty \}$ 
is again a Whitney function of order $m$.  
\vspace{2mm}

Let us now introduce some more notation used in this paper. As above let 
$Y\subset X \subset U \subset \R^n$ with $U$ being open and $Y$ and $X$ being
relatively closed in $U$. Denote by $\calJ (X,U) \subset \calC^\infty (U)$
or by $\calJ (X)$ if no confusion can arise
the ideal of all smooth functions on $U$ vanishing on $X$, and
let $\calJ^\infty (X,U) \subset \calC^\infty (U)$ be the ideal of 
smooth functions on $U$ which are flat on $X$, i.e.~let 
\[
  \calJ^\infty (X,U) := \{ f \in \calC^\infty (U) \mid (\partial^\alpha f)_{|X} = 0 
  \text{ for all $\alpha \in \N^n$}\}.
\] Then put 
\begin{eqnarray}
\label{Eq:DefSmFct}
  \calC^\infty (X)  & \!\! := \!\! & \calC^\infty (U) / \calJ (X,U) , 
  \quad \text{and} \\
\label{Eq:DefSmFctRel}
  \calC^\infty (X,Y)  &  \!\! := \!\! &
  \{f \in \calC^\infty (X)\mid \text{there exists $\tilde f \in 
  \calJ^\infty (Y,U)$ s.t.~$\tilde f_{|X} = f$} \}. 
\end{eqnarray}
One calls $\calC^\infty (X)$ the \textit{algebra of smooth functions on $X$}
and $\calC^\infty (X,Y)$ the \textit{ideal of smooth functions on $X$ 
flat on $Y$}. Note that neither $\calC^\infty (X)$ nor $\calC^\infty (X,Y)$ 
depend on the particular choice of the ambient $U$.
By Whitney's Extension Theorem one obtains 
\begin{equation}
  \label{Eq:WET}
  \calE^\infty (X) \cong \calC^\infty (U) / \calJ^\infty (X,U), \quad 
  \text{and} \quad
  \calJ^\infty (Y,X) \cong \calJ^\infty (Y,U) / \calJ^\infty (X,U) .
\end{equation}
We finally put
\begin{equation}
  \label{Eq:ISF}
  \begin{split}
  \calE^\infty (Y, X) & := \calC^\infty (X) / \calC^\infty (X,Y) = \\
  & = \calC^\infty (U) / \calJ (X)  \Big/ \calJ^\infty (Y,U) / 
  \calJ^\infty (Y,U) \cap \calJ (X). 
\end{split}
\end{equation}
We will call $\calE^\infty (Y, X)$ the \textit{algebra of Whitney functions 
on $Y$ induced by smooth functions on $X$}. Note that
\begin{equation}
\label{Eq:QR}
\begin{split}
  \calE^\infty (Y,X) =  \calE^\infty (Y ) / 
  \calJ (X)\cdot \calE^\infty (Y) . 
\end{split}
\end{equation}

\begin{remark}
  \begin{enumerate}
  \item  According to its definition above, the symbol $\calC^\infty (X,Y)$ denotes here the same algebra as in 
  \cite{BieMilPawCDF}. This notation differs though from the one used in \cite{BieSchwaCLDEF}.  
  \item  Note that  $\calC^\infty (X,Y)$ coincides with the quotient of $\calJ^\infty (Y,X)$ by the 
     ideal $\calJ^\infty (Y,X) \cap \big\{ F \in \calE^\infty (X) \mid F_0 = 0 \big\}$. Unless $X$ 
     is an open subset of the ambient euclidean space, this ideal is in general non-zero, hence   
     $\calC^\infty (X,Y)$ and  $\calJ^\infty (Y,X)$ do in general not coincide.      
  \end{enumerate}

\end{remark}

Next, we wish to examine the behavior of Whitney functions with respect to the pullback 
 under smooth maps. To this end, let  $\varphi:U \to V$ be a smooth map  between
open subsets $U\subset \mathbb R^n$ and $V\subset \mathbb R^m$. 
Furthermore, assume that $X\subset U$ 
  and $Y\subset V$ be relatively closed subsets such that $\varphi(X)\subset Y$. 
Now, for a jet  $F \in \sfJ^\infty (Y)$  we  define, mimicking the formula of Fa{\`a} di Bruno 
(cf.~Theorem \ref{FdBmulti}), its  pullback  
  $\varphi_{X,Y}^{\sharp} F = (\varphi_{X,Y}^{\sharp} F )_{\alpha \in \N^n} \in \sfJ^\infty (X)$ 
  as follows: 
  \begin{equation}
    \label{Eq:DefPullbackJet}
    \begin{split}
      (& \varphi_{X,Y}^{\sharp}  F )_\beta= \sum_{\lambda \in \Lambda_{n,m}(\beta)} \frac{\beta !}{\lambda !} \quad F_{\sum_{\alpha}\lambda_\alpha} \circ \varphi \quad \prod_{\alpha}\frac{\left(\partial^\alpha \varphi\right)^{\lambda_\alpha}}{(\alpha!)^{\sum_i \lambda_{i,\alpha}}}
    ,
  \end{split}
  \end{equation}
with the notation as in Theorem \ref{FdBmulti}.
  The following result is an immediate consequence of the formula by 
  Fa{\`a} di Bruno. 
\begin{theorem}
\label{Thm:PullbackJet}
  Let $U\subset \R^n$ and $V\subset \R^m$ be open, and assume that $X\subset U$ 
  and $Y\subset V$ are relatively closed. Let $\varphi: U \rightarrow V$ be a 
  smooth map such that $\varphi (X) \subset Y$. 
  Then the pullback map $\varphi^{\sharp}:=\varphi_{X,Y}^{\sharp} : \sfJ^\infty (Y) \rightarrow \sfJ^\infty (X)$,
  $F\mapsto \varphi_{X,Y}^{\sharp} F$ is a continuous linear map. 
  Moreover, it makes the following diagram commute:
  \begin{equation}
  \xymatrix{
      0\ar[r]& \calJ^\infty(Y,V)\ar[r]\ar[d]^{\varphi^*}
      &\calC^\infty(V)\ar[r]^{\sfJ^\infty_Y}\ar[d]^{\varphi^*}&
      \calE^\infty(Y) \ar[r]\ar[d]^{\varphi_{X,Y}^{\sharp}}&0\\
      0 \ar[r]&\calJ^\infty(X,U)\ar[r]& \calC^\infty(U) \ar[r]^{\sfJ^\infty_X} &
      \calE^\infty(X)  \ar[r]& 0 \:,
  }\label{DefPullbackWhi}
\end{equation}
  where $\varphi^* : \calC^\infty(V) \rightarrow \calC^\infty(U)$ is the pull-back
  $f \mapsto f\circ \varphi$.
\end{theorem}
\begin{proof}
  By definition, $\varphi^{\sharp}$ is linear. To check that  $\varphi^{\sharp}$ is continuous
  let $K \subset X$ be compact and $k\in \N$. Then 
  \begin{align}\label{SupnormEstimate}
 \nonumber  | \varphi^{\sharp}  F|_{K,k}  & \leq \left(\sup_{y\in \varphi (K) \atop |\beta| \leq k} | F_\beta (y) |\right) \!\! 
\quad \left(\sup_{x\in X}   \sum_{\lambda \in \Lambda_{n,m}(\beta)} {\beta !\over \lambda}\quad \prod_{\alpha}\frac{\left|\partial^\alpha \varphi(x)\right|^{\lambda_\alpha}}{(\alpha!)^{\sum_i \lambda_{i,\alpha}}}\right)\\
   & \leq C_{\varphi , K, k }  \: |F|_{\varphi (K) ,k} ,
  \end{align}
  where the constant $C_{\varphi , K, k }>0 $ depends only on $\varphi$, $K$ and 
  $k$.

  In order to prove continuity of $\varphi^{\sharp}$ we need a similar estimate for the remainder term. To this end we observe that for $x\in X$ and $k\in\N$
  \begin{align*}
  \sfR_x^k\circ \varphi^{\sharp}&=(\id-\widetilde \sfT_x^k)\circ \varphi^{\sharp}\\
  &=\varphi^{\sharp}-\sfJ^k\circ T^k_x \circ\varphi^{\sharp}=\varphi^{\sharp}-\sfJ^k\circ \varphi^*\circ T^k_{\varphi(x)}\\
  &=\varphi^{\sharp}-\varphi^{\sharp}\circ\sfJ^k\circ T^k_{\varphi(x)}=\varphi^{\sharp}\circ\sfR_{\varphi(x)}^k.
  \end{align*}

Let us assume for the moment that $\varphi(x)\ne \varphi(x')$. Then we have for $|\alpha|\le k$: 
  \begin{align}\label{Lipschitz}
  {\left(\sfR_x^k(\varphi^{\sharp}F)\right)_\alpha\over |x-x'|^{k-|\alpha|}}&= {|\varphi(x)-\varphi(x')|^{k-|\alpha|}\over |x-x'|^{k-|\alpha|}}{\left(\varphi^{\sharp}(\sfR_{\varphi(x)}^k F)\right)_\alpha\over |\varphi(x)-\varphi(x')|^{k-|\alpha|}}
  \end{align}
  The first term on the right hand side can be estimated by $L^{k-|\alpha|}$, where $L$ is a Lipschitz constant for $\varphi_{|K}$. On the other hand one can prove using an argument similar to \eqref{SupnormEstimate}
  \begin{align*} 
  \left|\left(\varphi^{\sharp}(\sfR_{\varphi(x)}^k F)(x')\right)_\alpha\right|&\le C_{\varphi,\varphi(K),|\alpha|}\sup_{\beta\le \alpha}\left| (\sfR_{\varphi(x)}^k F)_{\beta}(\varphi(x'))\right|\\
  & \le C_{\varphi,\varphi(K),|\alpha|}\sum_{\beta\le \alpha}\left| (\sfR_{\varphi(x)}^k F)_{\beta}(\varphi(x'))\right|
  \end{align*}
  We observe that $|\varphi(x)-\varphi(x')|\le \operatorname {diam}(\varphi(K))<\infty$ and  obtain 
  \begin{align}\label{PullBackIsContinuous}
\nonumber  {\left(\sfR_x^k(\varphi^{\sharp}F)\right)_\alpha\over |x-x'|^{k-|\alpha|}}&
   \le \:\sum_{\beta\le \alpha} L^{k-|\alpha|} C_{\varphi,\varphi(K),|\alpha|}\operatorname {diam}(\varphi(K))^{|\alpha|-|\beta|}\:{\left| (\sfR_{\varphi(x)}^k F)_{\beta}(\varphi(x'))\right|\over |\varphi(x)-\varphi(x')|^{k-|\beta|}}\\
   &\le |\!|F|\!|'_{\varphi(K),k} \sum_{\beta\le \alpha} L^{k-|\alpha|} C_{\varphi,\varphi(K),|\alpha|}\operatorname {diam}(\varphi(K))^{|\alpha|-|\beta|}.
  \end{align}
  
  In order to deal with the case when  $\varphi(x)=\varphi(x')$ we observe that $\sfR_{x}^k (\varphi^{\sharp} F)(x')=\varphi^{\sharp}\sfR_{\varphi(x)}^k F(\varphi(x'))=0$. This implies that the left hand side of \eqref{Lipschitz} vanishes for all $x\ne x'$. This means that \eqref{PullBackIsContinuous} holds in all cases when $x\ne x'$, implying that 
  \begin{align*}
 |\!|\varphi^{\sharp} F|\!|'_{K,k}&\le \widetilde C_{\varphi,\varphi(K),k} |\!|F|\!|'_{\varphi(K),k},\\
 \widetilde C_{\varphi,\varphi(K),k}&= \sup_{\alpha,|\alpha|\le k}\left(\sum_{\beta\le \alpha} L^{k-|\alpha|}\:C_{\varphi,\varphi(K),|\alpha|}\:\operatorname {diam}(\varphi(K))^{|\alpha|-|\beta|}\right).
\end{align*}
This completes the proof that $\varphi^{\sharp}$ maps Whitney functions to Whitney functions and that
$\varphi^{\sharp}:\calE^m(Y)\to\calE^m(X)$ is a continuous map between Fr\'echet spaces for $m\in\N\cup\{\infty\}$. 

It remains to show that the diagram \eqref{DefPullbackWhi} is commutative.
  The formula of Fa\`a di Bruno immediately entails that 
  $\varphi^* f \in \calJ^\infty (X,U)$ for all $f\in \calJ^\infty (Y,V)$. 
  Hence the left square  in the diagram above commutes. To prove that
  the right square commutes as well, one has to show that for an element 
  $f \in \calC^\infty (V)$ the relation
  \begin{equation}
  \label{Eq:CommPol} 
     \varphi^{\sharp} \circ \sfJ^\infty (f)  = \sfJ^\infty ( f\circ \varphi ) 
  \end{equation}
  holds true. But this is clear by the formula of Fa\`a di Bruno 
  and Eq.~\eqref{Eq:DefPullbackJet}.
  Hence the claim holds true. 
\end{proof}
\begin{remark}
Note that with the notation of Theorem \ref{WhitneyExtension} $\operatorname{res}^X_Y=\id^{\sharp}_{Y,X}$ where $\id:U\to U$ is the identity map.
\end{remark}


%
%
\section{The orbit type stratification and saturations}
\label{orbittypestratification}
In this paper, we follow Mather's concept of a stratification, 
cf.~\cite{MatSM,PflAGSSS}, which essentially is a local one and allows for a 
clear construction of the orbit type stratification 
of a $G$-manifold $M$.

\begin{definition}
By a \emph{decomposition} of a paracompact topological space $X$ with countable topology we understand a locally finite 
partition $\calZ$ of $X$ into locally closed subspaces $S$, 
called \emph{pieces}, such that each piece $S\in \calZ$ carries
the structure of a smooth manifold and such that the following \emph{condition of frontier} is satisfied:
\begin{enumerate}
\item[(CF)] 
  If $R,S \in \calZ$ are two pieces such that $R \cap \overline{S} \neq \emptyset$, then $R \subset \overline{S}$. 
  In this case we say that $R$ is \emph{incident} to $S$. 
\end{enumerate}
If $\calZ$ and $\calZ'$ are two decompositions of $X$, we say that 
$\calZ$ is \emph{coarser} than $\calZ'$, if every piece of $\calZ'$ is
contained in a piece of  $\calZ$. 

A \emph{stratification}  of $X$ (in the sense of Mather)
is a map $\calS$ which associates to each 
point $x\in X$ the set germ $\calS_x$ of a locally closed subset of $X$ 
such that the following axiom holds true:
\begin{enumerate}
\item[(ST)] For each $x \in U$ there exists an open neighborhood $U$ and a 
  decomposition $\calZ_U$ of $U$ such that for each $y\in U$ the set germ 
  $\calS_y$ coincides with the set germ of the unique piece $R_y \in \calZ_U$ 
  which contains $y$ as an element. 
\end{enumerate}
\end{definition}

A decomposition of $X$ obviously induces a stratification, namely the one 
which associates to each point the piece in which that point lies. 
The crucial observation from \cite[Prop.~1.2.7]{PflAGSSS} now is that a 
stratification $\calS$ of $X$ in the sense of Mather is always induced by a 
global decomposition of the underlying space, and that among those 
decompositions there is a coarsest one. We usually denote that 
decomposition by the same symbol $\calS$ as for the stratification, and call 
the pieces of that decomposition the \emph{strata} of the stratification. 

\begin{example}
  Let $M$ be an analytic manifold, and $Z \subset M$ a subanalytic subset. 
  Then $Z$ possesses a minimal stratification fulfilling Whitney's condition B. 
  See \cite{BieMilSSS,BekRSSS} for subanalytic sets and their stratifications, and \cite{PflAGSSS} for 
  details on the Whitney conditions.  
\end{example}

Now assume that $G$ is a compact Lie group acting on a smooth manifold $M$. 
Denote for every point $x\in M$ by $G_x$ the isotropy group of $x$. 
For every closed subgroup $H \subset G$ the set 
\[
 M_{(H)} := \{ x \in M \mid G_x \text{ is conjugate to } H \}
\]
then is a smooth submanifold of $M$, possible with varying dimensions of its 
connected components. Moreoever, the map $\calS$ which associates
to each $x\in M$ the set germ of the submanifold 
$M_{(G_x)}$ at $x$ is a stratification of $M$, which we call the 
\emph{orbit type stratification} of $M$. Since each of the sets $M_{(H)}$ is 
$G$-invariant, the orbit type stratification descends
to a stratification of the orbit space $X := M/G$. This stratification is
also called \emph{orbit type stratification}. See \cite[Sec.~4.3]{PflAGSSS} 
for details. 

\begin{example} Let $n\ge 2$ and consider the orthogonal group $G=\operatorname{O}_n$ acting $V=T^*\R^n=\R^n\times \R^n=\{(q, p)\mid q, p\in \R^n\}$ as the cotangent lift of the defining representation (i.e.~$G$ acts on $\R^n\times \R^n$ diagonally). There are three orbit type strata:
\begin{enumerate}
\item the stratum $V_{(G)}=\{0\}$ of points with isotropy group equal $G$, 
\item the stratum $V_{(\operatorname{O}_{n-1})}=\{(q,p)\ne (0,0) \mid q\:||\: p\}$ of points with isotropy conjugate to $\operatorname{O}_{n-1}$, and
\item the stratum $V_{(e)}=\{(q,p) \mid q \nparallel p\}$ of points with trivial isotropy.
\end{enumerate}
\end{example}

In the following we extend the definition of orbit type stratification to 
singular subsets of a $G$-manifold $M$. The action of a group element $g\in G$ on $x\in M$ will be denoted by
$g.x$. 

\begin{definition}
  Let $G$ be a compact Lie group acting smoothly on a manifold $M$. We will say that a relatively closed subset $Z$ 
  of some open $G$-invariant subset $U \subset M$ is \emph{stratified by orbit type}, if the following conditions 
  hold true:
  \begin{enumerate}[(OT1)]
  \item \label{ite:OT1} 
        If $S$ is a stratum of the stratification of $U$ by orbit type, then 
        the intersection $Z\cap S$ is an analytic submanifold of $U$, 
        possibly with  connected components having varying dimensions. 
  \item \label{ite:OT3} 
        If $x,y \in Z$ are points for which there is a group element $g\in G$ 
        with $g.x =y$,  then 
        \[
          g.\big( T_x (Z\cap S) \big) = T_y(Z\cap S) ,
        \] 
        where $S$ is the orbit type stratum in $U$ which contains $x$ and $y$.
  \end{enumerate}
\end{definition}

\begin{lemma}
Let $M$ be a $G$-manifold and $Z \subset M$ relatively closed  
in some $G$-invariant open subset $U\subset M$ be stratified by orbit type. 
Then the following holds true:
\begin{enumerate}[{\rm (OT1)}]
\setcounter{enumi}{2}
 \item \label{ite:OT2} 
        If $S$ is a stratum of the stratification of $U$ by orbit type, and $x \in Z\cap S$,
        then there is an open neighborhood $O\subset M$ of $x$ such that for all $y \in O\cap Z\cap S$
        \[
           \dim \big( T_y \calO_y \cap  T_y(Z\cap S) \big) = \dim \big( T_x \calO_x \cap  T_x(Z\cap S) \big) \ .
        \] 
        Hereby, $\calO_y$ denotes the orbit through $y$. 
\end{enumerate}
\end{lemma}

\begin{proof}
  The claim follows from (OT\ref{ite:OT3}) and the fact that 
  $g.(T_x \calO_x ) = T_y \calO_y$.
\end{proof}

\begin{proposition}
   Let $M$ be a $G$-manifold, $U\subset M$ be open and $G$-invariant, and assume that $Z \subset U$ is a 
   relatively closed subset which is stratified by orbit type. 
   Then $G.Z$ is stratified by orbit type as well.
   Moreover, assigning to each $x\in Z$ (resp.~$x\in G.Z$) the set germ $[Z\cap S]_x$ 
   (resp.~ $[(G.Z)\cap S]_x$), where $S$ is the orbit type stratum of $U$ containing $x$,
   provides a stratification of $Z$ (resp.~$G.Z$). These stratifications are $G$-invariant which means that
   $g [Z\cap S]_x  = [Z\cap S]_{gx}$ for all $x \in Z\cap S$ and $g\in G$ with $gx \in  Z\cap S$, and  
   $g [(G.Z)\cap S]_y  = [(G.Z)\cap S]_{gy}$ for all $y \in G.Z\cap S$ and $g\in G$.
\end{proposition}
\begin{proof}
   We first show that for every relatively closed $Z \subset U$ which is stratified by orbit type the assignment
   $Z \ni x \to \calS_x := [Z\cap S]_x$, where $S$ is the orbit type stratum of $U$ containing $x$, is 
   a stratification of $Z$. Let $x \in Z$. Choose an open neighborhood $O$ of $x$ according to   (OT\ref{ite:OT2}).
   Then, by (OT\ref{ite:OT1}), $O \cap Z\cap S$ is a smooth manifold of some fixed dimension, and 
   $[Z\cap S]_x =  [Z\cap S \cap O]_x$. Locally, the assignment $\calS$ comes from a decomposition of $Z$, since
   the ambient manifold is decomposed by orbit type strata. Moreover,  since the decomposition  of the ambient space 
   by orbit type strata is locally finite, the local decomposition of $\calZ$ inducing $\calS$ has to be locally finite, 
   too. Therefore, the assignment  $Z \ni x \to \calS_x $ is a stratification of $Z$, indeed. 
   (OT\ref{ite:OT3}) entails that the stratification $\calS$ is $G$-invariant.
 
   It remains to show that $G.Z$ is stratified by orbit types. To this end choose a 
   $G$-invariant riemannian metric on $M$ and note that $G.Z$ is a relatively closed 
   subset of $U$. Next consider an orbit type stratum $S$ of $M$, let $y\in G.Z \cap S$ 
   and choose $g\in G$ with $x = gy \in Z$. By $G$-invariance of $S$ we have 
   $x \in Z \cap S$. 
   Choose an open connected neighborhood $O$ of $x$ such that (OT\ref{ite:OT2}) holds true. Let 
   $\mathfrak{h} \subset \operatorname{Lie} (G)$ be the sub Lie algebra consisting of all 
   $\xi \in \operatorname{Lie} (G)$ such that $\xi_V (x) \in T_x (Z\cap S)$, where $\xi_V$ 
   denotes the fundamental  vector field of $\xi$. Let $\mathfrak{m}$ be a complement of $\mathfrak{h}$
   in $\operatorname{Lie} (G)$. By (OT\ref{ite:OT2}),  we can find, after possibly shrinking $O$,  
   connected open neighborhoods $O_1 \subset \mathfrak{m}$ and $O_2 \subset T_x(Z\cap S)$ of the origin 
   such that the function
   \[
     \psi : O_1 \times O_2 \to M, \quad (\xi,X) \mapsto (\exp\xi) . \exp(X)
   \] 
   is well-defined, is an open embedding, and has image $O$.  Put $O' := g^{-1}.O$. The map 
   $\chi : O' \to O_1 \times O_2 \subset \mathfrak{m} \times T_x(Z\cap S)$, $z \mapsto \psi^{-1} ( g . z)$ 
   then is a smooth chart of $(G.Z) \cap S$ around $y$. 
   Property  (OT\ref{ite:OT3}) entails that if we choose a different $g \in G$ with 
   $g.x \in Z$, we obtain another smooth chart of $(G.Z) \cap S$ around $y$ which is 
   $\calC^\infty$-equivalent to the first. This proves that $G.Z$ satisfies (OT\ref{ite:OT1}). 
   Finally, $(G.Z) \cap S$ fulfills (OT\ref{ite:OT3}), since 
   both $G.Z$ and $S$ are $G$-invariant.  
\end{proof}


%
%
\section{Invariant Whitney functions}
\label{invwhitney}

We assume that $G$ is a compact Lie group acting orthogonally on $V=\R^n$.
Our goal is to prove Theorem \ref{maindiag}, an analogue of the Theorem of Schwarz for 
Whitney functions $\calE^\infty(Z)$ along a locally closed subanalytic set $Z\subset V$, which is stratified by orbit type. 
Note that we do not assume $Z$ to be $G$-stable. In order to speak of invariance of such Whitney functions we make use of 
the language of groupoids. 

\begin{definition} 
Let $Z\subset V$ be an arbitrary subset. The \emph{restricted action groupoid} $\Gamma_Z:=(G\ltimes V)_{|Z}$ is defined as follows. The set of objects $(\Gamma_Z)_0$ is defined to be $Z$ while the set 
of arrows is  $(\Gamma_Z)_1:=\{(g,z)\in G\times Z\mid g.z\in Z\}$. The composition of arrows, 
unit and inverse are defined in the obvious way. That is, the composition rule is 
$(g,z)(h,z')=(gh,z')$ in case of $z=h.z'$, the unit map sends $z$ to $(e,z)$, while the 
inverse of $(g,z)$ is $(g^{-1},g.z)$.
\end{definition}

For $g\in G$ let us write $\Phi_g$ for the (linear) diffeomorphism $V\to V$, $v\mapsto g.v$. 
We have a corresponding formal pullback (cf. Eq. (\ref{Eq:DefPullbackJet}))
\begin{eqnarray*}
(\Phi_g)^{\sharp}_{\{v\},\{g.v\}}:\sfJ^\infty(\{g.v\})\to\sfJ^\infty(\{v\}).
\end{eqnarray*}
Given a locally closed subset $Z\subset V$ we say that a jet $F=(F_{\alpha})_{\alpha\in\mathbb N^n}\in \sfJ^\infty(Z)$ is 
\emph{invariant} if the following conditions hold true:
\begin{enumerate}[{(Inv}1)]
\item \label{Ite:Action}
  The canonical action of $(G\ltimes V)_{|Z}$ on the jet bundle $\sfJ^\infty(Z)$
  leaves $F$ invariant, which means that
  for all $(g,z)\in(G\ltimes V)_{|Z}$ and  $\alpha\in \mathbb N^n$ one has 
  \begin{eqnarray*}
     \left((\Phi_g)^{\sharp}_{\{z\},\{g.z\}}\big(F (g.z) \big)\right)_\alpha=F_{\alpha}(z)
  \end{eqnarray*}
\item\label{Ite:Constant}
     The natural action of the Lie algebra $\frg$ of $G$ on the jet bundle $\sfJ^\infty(Z)$
     leaves $F$ invariant, which means that for every element $\xi \in \frg$ one has 
     \[
       \xi_V  F =0,
     \]
    where $\xi_V$ denotes the fundamental vector field of $\xi$ on $V$.
\end{enumerate}

The space of jets on $Z$ satisfying invariance condition (Inv\ref{Ite:Action}) will be  
denoted by $\sfJ^\infty(Z)^{(G\ltimes V)_{|Z}}$, the space of jets satisfying invariance condition 
(Inv\ref{Ite:Constant}) by $\sfJ^\infty(Z)^\frg$. By $\sfJ^\infty(Z)^\textup{inv}$, we denote the 
space of invariant jets, i.e.~the space
\[
  \sfJ^\infty(Z)^\textup{inv} = \sfJ^\infty(Z)^{(G\ltimes V)_{|Z}} \cap \sfJ^\infty(Z)^\frg \: .
\]
For Whitney functions, we put
\[
\begin{split}
 & \calE^\infty(Z)^{(G\ltimes V)_{|Z}}:=\sfJ^\infty(Z)^{(G\ltimes V)_{|Z}}\cap \calE^\infty(Z), \\
 & \calE^\infty(Z)^\frg:=\sfJ^\infty(Z)^\frg\cap \calE^\infty(Z), \: \text{ and } \\
 & \calE^\infty(Z)^\textup{inv}:=\sfJ^\infty(Z)^\textup{inv}\cap \calE^\infty(Z).
\end{split}
\]
We call $\calE^\infty(Z)^\textup{inv}$ the space of \emph{invariant Whitney functions}.
Finally, if $M$ denotes a $G$-manifold, and $Z\subset M$ a closed subset,
we write $\calJ^\infty(Z,M)^G$ for the space $\calJ^\infty(Z,M) \cap \calC^\infty (M)^G$
and call it the space of \emph{invariant smooth functions on $M$ flat on $Z$}.
\begin{remark}
 In case $Z\subset V$ is locally closed and $G$-stable, the restricted groupoid 
 $(G\ltimes V)_{|Z}$ coincides with the action groupoid $G\ltimes Z$. For convenience, 
 we therefore write $\calE^\infty(Z)^G$  instead of $\calE^\infty(Z)^{(G\ltimes V)_{|Z}}$ 
 in this situation. Observe  that for Whitney functions over a $G$-stable 
 $Z$, condition (Inv\ref{Ite:Constant}) follows from (Inv\ref{Ite:Action}).  In other words 
 this means that $\calE^\infty(Z)^G = \calE^\infty(Z)^\textup{inv}$, if $G.Z =Z$.
 Note that this property is essentially  a consequence  of the Theorem of Schwarz and Mather.
 See Prop.~\ref{surgrp} for details. 
\end{remark}
\begin{example} To motivate that in case of non-$G$-stable $Z$ it is necessary to impose $\mathfrak g$-invariance in addition to the $(G\ltimes V)_{|Z}$-invariance let us consider the following example. We let the circle $G:=\operatorname S^1$ operate on the plane $V:=\R^2=\{(x,y)\mid x,y\in \R\}$ by rotations and put $Z:=\{(0,1)\}$. Since $Z$ consists of a point we have $\calE^\infty(Z)=\sfJ^\infty(Z)$. The restricted action groupoid is trivial: $(G\ltimes V)_{|Z}=\{(e,0,1)\}$, where $e\in \operatorname S^1$ is the identity. Hence 
 $\sfJ^\infty(Z)^{(G\ltimes V)_{|Z}}=\sfJ^\infty(Z)$, i.e., the condition of invariance with respect to the action groupoid is void in this example. On the other hand, the fundamental vector field of the circle action evaluated at the point $(0,1)$ is $-\partial/\partial x=-\partial_{(1,0)}$ which entails that $\sfJ^\infty(Z)^\mathfrak g=\{F=(F_\alpha)_{\alpha\in \N^2}\in \sfJ^\infty(Z)\mid \alpha \notin \{0\}\times\N \Rightarrow F_\alpha=0\}$. 
\end{example}

\begin{theorem} \label{maindiag} Let $G\to \operatorname{O}(V)$ be a finite dimensional 
orthongonal representation of the compact Lie group $G$.  Let 
$\rho_1,\dots,\rho_\ell\in\mathbb R[V]^G$  be a minimal complete system of polynomial 
invariants and $\rho:=(\rho_1,\dots,\rho_\ell): V\to \mathbb R^\ell$ the corresponding 
Hilbert map. Assume to be given a subanalytic subset $Z\subset V$ which is  stratified by orbit type 
and closed in some open $G$-invariant open neighborhood $U \subset V$. 
Let us write $\calJ_X:=\calJ(X,W)$ for the ideal of smooth functions on $W$ vanishing on $X:= \rho (U)$, where
$W \subset \R^\ell$ open can and has been chosen so that $X$ is closed in $W$.  
Then in the commutative diagram
\begin{eqnarray*}
\xymatrix{&0&0&0&\\
0\ar[r]&\calJ ^\infty(Z,U)^G\ar[u]\ar[r]&\calC^\infty(U)^G\ar[u]\ar[r]_{\sfJ^\infty_{Z}\quad\quad}&\calE^\infty(Z)^\textup{inv}\ar[u]\ar[r]&0\\
0\ar[r]&\calJ ^\infty(\rho(Z), W)\ar[u]^{\rho^*_{|U}}\ar[r]&\calC^\infty(W)\ar[u]^{\rho^*_{|U}}\ar[r]_{\sfJ^\infty_{\rho(Z)}\quad}&\calE^\infty(\rho(Z))\ar[u]_{\rho^{\sharp}_{Z,\rho(Z)}}\ar[r]&0\\
0\ar[r]&\calJ ^\infty(\rho(Z), W)\cap \calJ_X \ar[u]\ar[r]& \calJ_X \ar[u]\ar[r]&\sfJ^\infty_{\rho(Z)}(\calJ_X)\ar[u]\ar[r]&0\\
&0\ar[u]&0\ar[u]&0\ar[u]&
}
\end{eqnarray*}
all rows and columns are exact sequences of linear continuous maps of Fr{\'e}chet spaces.
\end{theorem}

\begin{remark}
 Let us note that as  consequence of $\rho^{\sharp}_{Z,{\rho(Z)}}(\calE^\infty(\rho(Z)))=\calE^\infty(Z)^\textup{inv} $ one can conclude that 
 $\calE^\infty(Z)^\textup{inv}\subset \calE^\infty(Z)$ is a closed subspace. Here we make use of \cite[Theorem 3.6]{BieMilCDF} and the 
 fact that $\rho:U\to X$ is regular in the sense of Gabrielov by Prop.~\ref{Prop:GabrielovRegularityPolynomials}. Moreover, 
 occasionally $\sfJ^\infty_Z$ admits a continuous split (this is the case if and only if the interior of $Z$ is dense in $Z$, 
 cf.~\cite{BieEWFSS}). In this situation, employing the split of $\rho_{|U}^*$ (see Theorem \ref{SchwaSFIACLG}) we can conclude that 
 $\rho^{\sharp}_{Z,{\rho(Z)}}$ is split surjective. It is not known to the authors under what condition on $Z$ the map 
 $\rho^{\sharp}_{Z,\rho(Z)}$ is actually split. Moreover, it is unclear if the image of 
 of $\rho^{\sharp}_{Z,{\rho(Z)}}$ is $\calE^\infty(Z)^\textup{inv}$ with weaker assumptions on $Z\subset V$.
\end{remark}


Before we turn to the proof of Theorem \ref{maindiag} we need a couple of auxiliary results.

\begin{proposition} \label{Prop:TopProp}
Let $G$, $V$ be as in  Theorem  \ref{maindiag},  $\rho : V \rightarrow \R^\ell$ a Hilbert map,
and $Z\subset V$ closed in some $G$-invariant open set $U\subset V$. Then the following holds true: 
\begin{enumerate}
\item \label{It:Closed} $G.Z$ is closed in $U$.
\item \label{It:QuotientTop} 
      $G.Z$ carries the quotient topology with respect to the restricted action 
      \[ \Phi_Z:G\times Z\to G.Z,\: (g,z) \mapsto g.z .\]
\item \label{ite:PropernessRestriction}
      For every open  $W \subset \R^\ell$ with $\rho (U) = \rho (V) \cap W $ the Hilbert map 
      $\rho : V \rightarrow \R^\ell$ restricts to a proper map 
      $\rho_{|U} : U \rightarrow W$.
\item \label{It:ExOpNbhd}
      There exists an open set $W \subset \R^\ell$ such that 
      $\rho (U) = \rho (V) \cap W $. If $\rho (U)$ is semialgebraic (resp.~subanalytic), $W$ can be chosen 
      to be semialgebraic (resp.~subanalytic) as well. 
      In both cases $\rho (U)\subset W$ is Nash subanalytic. 
\end{enumerate}
\end{proposition}

 \begin{proof} Claim (\ref{It:Closed}) follows immediately by compactness of $G$.
   %
   %
   To prove (\ref{It:QuotientTop}), let $A\subset G.Z$ be such that   $\Phi_Z^{-1}(A)$ is a 
   closed subset of $G\times Z$. We have to show that $A$ is closed in $G.Z$.
   To this end consider a sequence of points $ g_k.z_k \in A$, where $k\in \N$, $g_k \in G$, and $z_k\in Z$.
   Assume that $(g_k.z_k)$ converges to some $a\in G.Z$. By compactness of $G$ one concludes that after 
   possibly passing to a subsequence, $(g_k)$ converges to some $g\in G$. Then $(z_k)$ converges to a
   point $z:= g^{-1} a$. Since $a\in U$, and $U$ is $G$-invariant, one has $z\in U$. Hence $z\in Z$ since
   $Z$ is closed in $U$.  By assumption, $\Phi_Z^{-1}(A)$ is closed in $G\times Z$, and 
   $(g_k,z_k) \in\Phi_Z^{-1}(A)$ for all $k\in \N$. But then $(g,z) \in \Phi_Z^{-1}(A)$, hence 
   $a = g.z \in A$ which shows (\ref{It:QuotientTop}).
  
   For every open $W \subset \R^\ell$ such that $\rho (U) = \rho (V) \cap W $ properness of 
   the restricted map $\rho_{|U} : U \to W$ follows immediately from $\rho$ being proper. 
   This gives (\ref{ite:PropernessRestriction}).
   
  To prove (\ref{It:ExOpNbhd}), observe that $\rho$ is a closed map since 
  since $\R^n$ is locally compact and $\rho$ a proper map. Hence 
  $W := \R^\ell \setminus \rho (V\setminus U)$ is open, and contains $\rho (U)$ as a 
  subset since $ \varrho (V\setminus U)  =\rho (V) \setminus \rho (U)$ by $G$-invariance of $U$
  and by the fact that $\rho$ is a Hilbert map. 
  By construction, $W$ is semialgebraic (resp.~subanalytic), if $\rho (U)$ is. 
  Since $U$ is $G$-invariant and since $\rho$ factors through an injective map on the orbit space $V/G$ 
  the equality  $\rho (U) = W \cap \rho (V)$ holds true.
  Finally, $\rho (U) \subset W$ is Nash subanalytic in the sense of Def.~\ref{Def:NashSubanalytic}
  since $\rho$ is Gabrielov regular by Prop.~\ref{Prop:GabrielovRegularityPolynomials}
  and since  $\rho_{|U} : U \to W$ is proper. 
\end{proof}

\begin{proposition}
\label{BSrefinement}
Let $G$ be a compact Lie group, $M$ a smooth $G$-manifold and $Z\subset M$ a 
closed subset. Then 
\begin{eqnarray}\label{Hsaturation}
  \calJ^\infty(Z,M)^G=\calJ^\infty(G.Z,M)^G.
\end{eqnarray}
\end{proposition}

\begin{proof} In order to prove Eq.~(\ref{Hsaturation}), it suffices to prove the 
inclusion $\subset$ since the converse inclusion is trivial. 
For $g\in G$ we denote the $g$-action on $M$ by $\Phi_g : M \to M$, $x \mapsto g.x $. 
The idea of the proof is to take partial derivatives of 
\begin{equation}\label{Ginv}
  f\circ\Phi_g=f
\end{equation}
and evaluate the result at $z\in Z$. Using local coordinates $(x^1,\dots,x^m)$ 
around $z$ and $(y^1,\dots,y^m)$ around $g.z$, and taking first order partial 
derivatives we obtain
\[
  \sum_j\frac{\partial f}{\partial y^j}(g.z)\frac{\partial \Phi^j_g}{\partial x^i}(z)=
  \frac{\partial f}{\partial x^i}(z).
\]
We see that $\frac{\partial f}{\partial x^i}(z)=0$ is equivalent to 
$\frac{\partial f}{\partial y^j}(g.z)=0$, since the Jacobi matrix of $\Phi_g$ is invertible. 
Let us assume that $f$ is flat on $Z$ and continue inductively. By induction hypothesis 
\begin{equation}\label{lot}
  \frac{\partial^k f}{\partial y^{j_1}\dots \partial y^{j_k}}(g.z)=0 \mbox{ for }k\le n-1.
\end{equation}
Taking $n$-th partial derivatives of equation (\ref{Ginv}) we find (cf. Theorem \ref{FdBcomb})
\begin{eqnarray*}
  \sum_{j_1,\dots,j_n}\frac{\partial^n f}{\partial y^{j_1}\dots \partial y^{j_n}}(g.z)
  \frac{\partial \Phi^{j_1}_g}{\partial x^{i_1}}(z)\cdots
  \frac{\partial \Phi^{j_n}_g}{\partial x^{i_n}}(z)+\text{lower order terms}
  =\frac{\partial^n f}{\partial x^{i_1}\dots \partial x^{i_n}}(z).
\end{eqnarray*}
Here, by lower oder terms we mean terms containing  a factor of the form (\ref{lot}). 
It follows again from the invertibility of the Jacobian of $\Phi_g$, that 
$\frac{\partial^n f}{\partial y^{j_1}\dots \partial y^{j_n}}(g.z)=0$, which proves 
Eq.~(\ref{Hsaturation}).
\end{proof}

\begin{proposition} \label{flatalongZ} Under the assumptions of Theorem \ref{maindiag}
the image of the pullback 
$\rho^*:\calJ^\infty(\rho(Z),W)\to \calJ^\infty(Z,U)$ is actually $\calJ^\infty(Z,U)^G$.
\end{proposition}

\begin{proof} First of all, a function of the form $f\circ \rho\in \calC^\infty(U)$ is clearly 
invariant. Moreover, if $f$ is flat along $\rho(Z)$, then by the formula of Fa{\`a} di Bruno 
(cf. Theorem \ref{FdBcomb}), $f\circ \rho$ is flat along $Z$.

To prove the converse, recall that $X = \varrho (U)\subset W$ is Nash subanalytic by 
Prop.~\ref{Prop:TopProp} (\ref{It:ExOpNbhd}), which 
by the work of Bierstone and Milman  \cite[Thm.~3.2]{BieMilCDF} \& \cite[Thm.~1.13]{BieMilGDPSS} entails that 
the pair $\rho(Z)\subset X$ of subanalytic sets has the 
\emph{composite function property}.  That is, for any proper real-analytic map 
$\varphi:M\to  W$ from a real-analytic manifold $M$ such that 
$\varphi(M)=X$ one has
\begin{equation}\label{composite}
  \varphi^*(\calJ^\infty(\rho(Z), W))=
  \left( \varphi^*(\calJ^\infty(\rho (Z), W))\right)^{\wedge}.
\end{equation}

Let us explain what the right hand side of Eq.~(\ref{composite}) means. 
A smooth function $f$ on $M$ is \emph{formally a composite with} 
$\varphi$ if for any $b\in \varphi(M)$ there exists an  
$F_b\in \widehat{\mathcal O}_b$ 
such that for all $a\in \varphi^{-1}(b)$
\[
   \widehat{f}_a=F_b\circ \widehat{\varphi}_a,
\]
where $\widehat{f}_a$ and $\widehat{\varphi}_a$ are the formal Taylor expansions at 
$a$ of $f$ and $\varphi$, respectively. The set of formally composite functions 
with $\varphi$ is denoted by $\big(\varphi^*\calC^\infty(W)\big)^\wedge$. 
Setting $Y:=\rho(Z)$, the space $\big(\varphi^*\calJ^\infty(Z,\mathbb R^\ell)\big)^\wedge$ is 
defined as the intersection 
$\left(\varphi^*\mathcal C^\infty(W)\right)^\wedge 
 \cap\calJ^\infty(\varphi^{-1}(Y),M)$.

Recall from above that the restricted Hilbert map $\rho_{|U} : U \to W$ is proper. Hence we may 
specialize Eq.~(\ref{composite}) to the case $\varphi=\rho_{|U}$ and conclude, remembering $\rho^{-1}(\rho(Z))=G.Z$, that
\begin{eqnarray*} \label{lhs}
  \rho_{|U}^*(\calJ^\infty(\rho (Z), W ))&=&
  \big( \rho_{|U}^*(\calJ^\infty(\rho(Z),W ))\big)^{\wedge}\\
  &=&\left(\rho_{|U}^*\calC^\infty( W )\right)^{\wedge}\cap 
  \calJ^\infty\left(\rho^{-1}(\rho(Z)),U\right)\\
   &=&\ \rho_{|U}^*\calC^\infty( W )\cap 
  \calJ^\infty\left(\rho^{-1}(\rho(Z)),U\right)\\
   &=&\ \calC^\infty(U)^G \cap \calJ^\infty(G.Z,U) \ .
\end{eqnarray*}
 Using Theorem \ref{SchwaSFIACLG} on differentiable invariants we see 
that $\rho^*(\calJ^\infty(\rho (Z), W))$ coincides with $\calJ^\infty(G.Z,U)^G$, which is  
$\calJ^\infty(Z,U)^G$ by Proposition \ref{BSrefinement}.
\end{proof}

\begin{proposition} \label{surgrp} Assume that $G$, $V$ and $\rho$ are as in Theorem \ref{maindiag} 
  and that $Z$ is a closed subset of an open $G$-invariant subset $U\subset V$. 
  Then the following statements hold true.
\begin{enumerate} 
\item \label{surInvRels}
  Over the $G$-stable set $G.Z$, the relation
  $\calE^\infty(G.Z)^G = \calE^\infty(G.Z)^\textup{inv}$ holds true.
\item \label{surJgrp} 
 The Taylor morphism $\sfJ^\infty_{G.Z}: \calC^\infty(U)^G\to \calE^\infty(G.Z)^G$ is onto 
  with kernel $\calJ^\infty(G.Z,U)^G$. 
\item \label{surpbgrp} 
  The image of the pullback $\rho^{\sharp}_{G.Z,\rho(Z)}:\calE^\infty(\rho(Z))\to \calE^\infty(G.Z)$ 
  is $\calE^\infty(G.Z)^G$.
\item \label{imJgrpd} 
  The Taylor morphism $\sfJ^\infty_{Z}: \calC^\infty(U)^G\to \calE^\infty(Z)$ has image 
  in the space of invariant Whitney functions $\calE^\infty(Z)^\textup{inv}$.
\item \label{impbgrpd} 
  The image of the pullback $\rho^{\sharp}_{Z,\rho(Z)}:\calE^\infty(\rho(Z))\to \calE^\infty(Z)$ is 
  contained in  the space $\calE^\infty(Z)^\textup{inv}$. 
\end{enumerate}
\end{proposition}

\begin{proof} Take an $F\in\calE^\infty(G.Z)^\textup{inv}$. By Whitney's extension theorem 
there exists an $f\in \calC^\infty(U)$ such that $ \sfJ^\infty_{G.Z}(f)=F$.
Taking the average with respect to the Haar measure on $G$ and using the equivariance of 
the Taylor morphism, we obtain (cf.~Theorem \ref{Thm:PullbackJet})
\begin{eqnarray*}
\sfJ^\infty_{G.Z}(\operatorname{Average}(f))=\sfJ^\infty_{G.Z}\left(\frac{1}{\operatorname{vol}(G)}\int_G g^*(f) \,
\mathrm{d}g  \right)=\frac{1}{\operatorname{vol}(G)}\int_G g^{\sharp}(F)\, \mathrm{d}g =F.
\end{eqnarray*}
This proves (\ref{surJgrp}), and also (\ref{surInvRels}), since 
$\operatorname{Average}(f)$ is constant along orbits, which means that 
$\xi_U\operatorname{Average}(f)=0$ for all $\xi\in \frg$.

By the Theorem of Schwarz and Mather \ref{SchwaSFIACLG} we find an 
$h\in \calC^\infty(W)$ such that $\operatorname{Average}(f)=\rho_{|U}^*(h)$. By Theorem \ref{Thm:PullbackJet} we have
\[F=(\sfJ^\infty_{G.Z}\circ \rho_{|U}^*)(h)=(\rho_{|U}^{\sharp}\circ\sfJ^\infty_{\rho(Z)})(h), \]
which proves (\ref{surpbgrp}).

To prove \ref{imJgrpd}, fix $f\in  \calC^\infty(U)^G$, and let $(g,z) \in (G\ltimes V)_{|Z}$. By $G$-invariance of 
$f$, one concludes that 
\[
  (\Phi_g)^{\sharp}_{\{z\},\{g.z\}}\big( \sfJ^\infty_{g.z} (f) \big) = \big (\sfJ^\infty_{z} \circ \Phi_g^* \big) (f) =
  \sfJ^\infty_{z} (f) ,
\]
hence  $\sfJ^\infty_{Z} (f)$ is invariant under the groupoid action $(G\ltimes V)_{|Z}$ . 
In addition, $\sfJ^\infty_{Z} (f)$ is invariant under the $\frg$-action, since 
$f$ is constant along orbits, hence $\xi_U f =0$ for all $\xi\in \frg$.
One now shows (\ref{impbgrpd})  analogously to (\ref{surpbgrp}), and the proposition is proved.
\end{proof}

Note, that if $Z$ is not $G$-saturated the Taylor morphism $\sfJ^\infty_{Z}$ is not $G$-equivariant, and we 
can not produce groupoid invariant Whitney functions by averaging over $G$. 
We will use Theorem 3.6 from Bierstone and Milman's article \cite{BieMilCDF} to prove 
the surjectivity for the groupoid case. The following result is a crucial tool  for this argument.

\begin{proposition} \label{idiso} 
Let $G$, $V$ be as above, and $Z\subset V$ a locally closed subset. Then the  map 
$\id^{\sharp}_{Z,G.Z}: \sfJ^\infty(G.Z)^G\to \sfJ^\infty(Z)^{(G\ltimes V)_{|Z}}$ 
is an isomorphism of Fr{\'e}chet algebras.
\end{proposition}

\begin{proof} Let us write for the action of $g\in G$: $v\mapsto g.v=\Phi_g(v)$. Let us also use the short 
hand notation $\Gamma_Z:= (G\ltimes V)_{|Z}$.

Our aim is to write down an inverse map for the restriction of $\id^{\sharp}_{Z,G.Z}$ to the space of invariant 
jets on $G.Z$. So let us pick a $\Gamma_Z$-invariant jet $F=(F_\alpha)_{\alpha\in \mathbb N^n}$ on $Z$, a point 
$v=g.z\in G.Z$, and declare 
$\widetilde F (v) = \big( \widetilde F_{\alpha} (v)\big)_{\alpha\in \mathbb N^n}  \in \sfJ^\infty(\{v\})$ 
using the pullback with respect to the diffeomorphism $\Phi_{g^{-1}}$:
\begin{eqnarray}\label{invextension}
\widetilde F(v) :=\Phi^{\sharp}_{g^{-1}} \big(  F ( z) \big) = (\Phi^{\sharp}_{g^{-1}} F)(g.z) \: .
\end{eqnarray}
Let us prove that this definition does not depend on the choice of $z$ and $g$.
To this end, assume that $g.z=g'.z'$ for some $g'\in G$ and $z'\in Z$. Setting $h:=(g')^{-1}g$ we have that $h.z=z'$, that is, $(h,z)\in\Gamma_Z$. With these preparations we are ready to calculate
\begin{eqnarray*}
\widetilde F(g'.z')&=&(\Phi^{\sharp}_{g'^{-1}} F)(g'.z')=(\Phi^{\sharp}_{hg^{-1}} F)(g.z)\\
&=&\left(\Phi^{\sharp}_{g^{-1}}(\Phi^{\sharp}_{h}F)\right)(g.z)\stackrel{*}{=}(\Phi^{\sharp}_{g^{-1}}F)(g.z)=\widetilde F(g.z),
\end{eqnarray*}
where we have used at``$*$'' that, by assumption, $F$ is $\Gamma_Z$-invariant.
As in the defining equation (\ref{invextension}) $\widetilde F (g.z)$ clearly depends continuously on 
$z$ and $g$, and $G\times Z \rightarrow G.Z$ is a topological identification by (\ref{It:QuotientTop}) in 
Prop.~\ref{Prop:TopProp},
$\widetilde F=(\widetilde F_{\alpha})_{\alpha\in \mathbb N^n}$  actually defines a jet in $\sfJ^\infty(G.Z)$.

It remains to check that $\widetilde F=(\widetilde F_{\alpha})_{\alpha\in \mathbb N^n}$ is $G$-invariant. 
To this end let $h \in G$, and $v =g.z\in G.Z$. Then 
\begin{eqnarray*}
\Phi^{\sharp}_{h} \big( \widetilde F(h.v) \big) &=& \Phi^{\sharp}_{h} \left( \Phi^{\sharp}_{(hg)^{-1}} \big( F (z)\big)\right) 
= \Phi^{\sharp}_{h} \left( \Phi^{\sharp}_{h^{-1}}\Phi^{\sharp}_{g^{-1}} \big( F (z) \big) \right) \\
&=& \Phi^{\sharp}_{g^{-1}} \big( F (z) \big)=  \widetilde F(v) ,
\end{eqnarray*}
which means that $ \widetilde F$ is $G$-invariant. 

Finally, we have to convince ourselves that the linear map 
\begin{eqnarray}
\qquad\sfE_Z : \sfJ^\infty(Z)^{(G \ltimes V)_{|Z}}\to \sfJ^\infty(G.Z)^G,\quad 
  F=(F_\alpha)_{\alpha\in \mathbb N^n}\mapsto\widetilde F=(\widetilde F_{\alpha})_{\alpha\in \mathbb N^n}
\end{eqnarray} 
is actually a  continuous right inverse to the injective map 
$\id^{\sharp}_{Z,G.Z}: \sfJ^\infty(G.Z)^G\to \sfJ^\infty(Z)^{(G \ltimes V)_{|Z}}$, but this is clear by construction.
The claim follows.
\end{proof} 
We now come to the main result which will be needed for proving Theorem \ref{maindiag}. 
\begin{proposition} \label{surgrpd} 
Under the assumptions of Theorem \ref{maindiag}  the following statements hold true.
\begin{enumerate}
\item \label{surJgrpd} The Taylor morphism $\sfJ^\infty_{Z}: \calC^\infty(U)^G\to 
\calE^\infty(Z)^\textup{inv}$ is onto with kernel 
\[
 \calJ^\infty(Z,U)^G=\calJ^\infty(G.Z,U)^G .
\]
\item \label{surpbgrpd} The image of the pullback 
$\rho^{\sharp}_{Z,\rho(Z)}:\calE^\infty(\rho(Z))\to \calE^\infty(Z)$ 
is actually the space of invariant Whitney functions $\calE^\infty(Z)^\textup{inv}$.
\end{enumerate}
\end{proposition}

\begin{proof} 
Consider the isomorphism 
$\id^{\sharp}_{Z,G.Z}:\sfJ^\infty (G.Z)^G \rightarrow \sfJ^\infty (Z)^{(G \ltimes V)_{|Z}}$ 
from Proposition \ref{idiso}, and its inverse $\sfE_Z$. Assume that we can show that                  
$\sfE_Z$ maps the invariant space $\calE^\infty (Z)^\textup{inv}$ to $\calE^\infty (G.Z)^G$. 
Since the image of 
$\calE^\infty (G.Z)^G$ under $\id^{\sharp}_{Z,G.Z}$ obviously lies in $\calE^\infty (Z)^\textup{inv}$
it then follows that $\id^{\sharp}_{Z,G.Z}$ restricts to an isomorphism 
$\iota_{Z} : \calE^\infty (G.Z)^G \rightarrow \calE^\infty (Z)^\textup{inv}$. 
In view of the commutativity of the diagram
\begin{eqnarray*}
 \xymatrix{\calC^\infty(U)^G\ar[r]^{\sfJ^\infty_Z\quad}\ar[dr]_{\sfJ^\infty_{G.Z}}&\calE^\infty(Z)^\textup{inv}\\
&\calE^\infty(G.Z)^G\:,\ar[u]_{\iota_{Z}}}
\end{eqnarray*}
statement (\ref{surJgrpd}) then is a consequence of Proposition \ref{surgrp}(\ref{surJgrp}).
Similarly, statement (\ref{surpbgrpd}) follows from the commutativity of the diagram
\begin{eqnarray*}
 \xymatrix{\calE^\infty(\rho(Z)) \ar[r]^{\rho^{\sharp}_{Z,\rho (Z)}}\ar[dr]_{\rho^{\sharp}_{G.Z,\rho (Z)}}&\calE^\infty(Z)^\textup{inv}\\
&\calE^\infty(G.Z)^G\ar[u]_{\iota_{Z}}}
\end{eqnarray*}
together with
Proposition \ref{surgrp}(\ref{surpbgrp}), if $\iota_Z$ is an isomorphism.

To prove the proposition, we therefore only need to show that $\sfE_Z$ maps the invariant space 
$\calE^\infty (Z)^\textup{inv}$ to $\calE^\infty (G.Z)^G$ or in other words that $\iota_Z$ is onto.
We will prove the claim by induction on the depth of the stratification 
of $Z$. Recall that by assumption, $Z$ is stratified by orbit types. Let us divide the argument in 
several steps.

{Step 1.} Denote by  $d$ the depth of the stratified space $Z$, and assume that for 
every subanalytic set $Y\subset V$ of depth less than $d$ such that $Y$ is closed in some $G$-invariant 
open subset of $V$ the map 
$\iota_{Y} : \calE^\infty (G.Y)^G \rightarrow \calE^\infty (Y)^\textup{inv}$ is an isomorphism. 
Let $S \subset Z$ be the stratum of $Z$ of highest depth. Then $S$, and hence $G.S$ are 
closed in $U$. Assume that we can show
\begin{enumerate}[({Claim} A)]
\item For every $z\in S$ there is an  $G$-invariant open semialgebraic neighborhood $O\subset U$ 
      such that $\sfE_{S\cap O}$ maps $\calE^\infty (S \cap O)^\textup{inv}$ to $\calE^\infty (G.S \cap O)^G$.
\end{enumerate}
Then, using a $G$-invariant locally finite partition of unity subordinate to the neighborhoods $O$ 
from Claim A, one derives that for every $F \in \calE^\infty (Z)^\textup{inv}$ the jet 
$\sfE_{S} \big(\operatorname{res}^Z_S F \big)$  is an invariant Whitney function on $G.S$, i.e.~an element of
$\calE^\infty (G.S)^G$. Now choose a $G$-invariant smooth function $f: U \rightarrow \R$ such that
$\sfJ^\infty_{G.S} (f) = \sfE_{S} \big(\operatorname{res}^Z_S F \big)$. Then $\widetilde F := F - \sfJ^\infty_{Z} (f)$ is a Whitney function on 
$Z$ such that $\widetilde F (z) =0$ for all $z\in S$. 
Now consider 
\begin{enumerate}[({Claim} B)]
\item Let $ E \in \calE^\infty (Z)^\textup{inv}$ such that $E (z) =0$ for all $z\in S$
  and such that $\sfE_{Z\setminus S} (\operatorname{res}^Z_{Z\setminus S} E) \in \calE^\infty (G. Z\setminus G.S )^G$. 
  Then $\sfE_Z (E)$ is an invariant Whitney function on $G.Z$. 
\end{enumerate}
Under the assumption that Claim B holds true, it then follows that $\sfE_Z (\widetilde F)$ is invariant, hence
\[
   \sfE_Z ( F )  = \sfE_Z (\widetilde F) + \sfJ^\infty_{G.S} (f)  \in \calE^\infty (G.Z)^G \: .
\]
But this finishes the inductive step, hence the proposition follows.  
It therefore remains to show Claim A and Claim B. 
In Step 2, we provide an auxiliary result which in Step 3  will be used to derive Claim A.
Claim B will be shown in Step 4.

{Step 2.} Fix a point $z$ in the stratum $S$ of $Z$ of highest depth. 
Denote by $\calO$ the $G$-orbit through $z$. Then choose $\varepsilon >0 $ such that 
the tubular neighborhood
\[
  T_\calO:= \{ v \in V \mid d(v,\calO) < \varepsilon \} 
\]
is contained in $U$ and such that the exponential map 
\[ \exp:TV \cong V\times V\rightarrow V, \: (v,w) \mapsto v+w \]
maps a $G$-invariant neighborhood $\widetilde O$ of the zero section of the normal bundle 
$\pi : N \rightarrow \calO$ equivariantly  onto $T_\calO$. 
 
Next, observe that the $G$-action on the normal bundle $N$ induces an action of the 
transformation groupoid 
$G\ltimes \calO$  on $N$ (see \cite[Sec.~3]{PflPosTanGOSPLG}). This means, that for every pair 
$(h,v) \in G \times \calO$ there is a (linear and isometric) isomorphism 
$\Psi_{(h,v)} : N_v \rightarrow N_{h.v}$ such that $\Psi_{(h',hv)} \Psi_{(h,v)}= \Psi_{(h'h,v)}$ for all
$h',h\in G$ and $v\in \calO$. In particular, if $hz = h'z$, then $ h'^{-1}h \in G_z$, and  
$\Psi_{h'^{-1}h,z} = \Psi_{(h'^{-1},hz)} \Psi_{(h,z)} $ is the standard action of  $ h'^{-1}h$ on $N_z$.
In other words, this means that $\Psi$ extends the natural action of the isotropy group $G_z$
on the normal space $N_z := T_z\calO^\perp$.  

After these preliminaries we now want to construct a smooth map $\lambda :N \rightarrow \R^k$
for some $k\in \N$ such that the following properties hold true:
\begin{enumerate}[(i)]
\item \label{Ite:SecProj}
      The map $ \lambda$ is $G$-invariant.
\item \label{Ite:GabProp}
      The map $\lambda$ is Gabrielov regular at each point of its domain 
      (see Appendix \ref{App:Gabrielov} or \cite{PawGRCAM} for Gabrielov regularity).
\item \label{Ite:HilMap} For every $v\in \calO$, the
  restriction $\lambda_v := \lambda_{|N_v} : N_v\rightarrow \R^k$ 
  is a minimal Hilbert map for the $G_v$-representation space $N_v$.
\item \label{Ite:FormComp}
      For every $v\in \calO$ and $n\in N_v^{G_v}$, a formal power series 
      $F\in \sfJ^\infty (\{n\})$ is invariant if and only if it is a composition with 
      $\lambda^{\sharp}_{\{n\},\{\lambda(n)\}} $.  
\end{enumerate}

To define $\lambda$, choose a minimal complete set of $G_z$-invariant polynomials
$\lambda_{1},\dots,\lambda_{k}$ on $N_z$, and denote by $\lambda_z : N_z \rightarrow \R^k$ 
the corresponding Hilbert map. 
Obviously, $\lambda_z$ is polynomial, hence according to 
Proposition \ref{Prop:GabrielovRegularityPolynomials}
$\lambda_z$ is regular in the sense of Gabrielov at each 
point of $N_z$.
Next we construct  $\lambda_v :N_v \rightarrow \R^k$ for arbitrary $v\in \calO$. 
To this end, choose an $h\in G$ such that $hv =z$. Then put
$\lambda_v (n) = \lambda_z \circ \Psi_{(h,v)} (n)$ for $n\in N_v$.
Now we can define:
\[
 \lambda : N \rightarrow \R^k, \quad n \mapsto  \lambda_{\pi (n)} (n) .  
\]
We have to show that each $\lambda_v$ does not depend on the particular choice of $h$,
and that $\lambda$ is  smooth. So let $h'\in G$ be another group element with 
$h'v =z$. Then 
\[
 \lambda_z \circ \Psi_{(h',v)} = \lambda_z \circ \Psi_{(h'h^{-1},z)} \circ \Psi_{(h,v)} =
  \lambda_z \circ \Psi_{(h,v)} ,
\]
since $h'h^{-1}\in G_z$, and $\lambda_z$ is $G_z$-invariant. Therefore $\lambda_v$
does not depend on the particular choice of $h\in G$ with $hv=z$. 
Next observe that for such $v$ and $h$ one can find  an open neighborhood $Q\subset \calO$ of $v$ and
an embedding $\sigma : Q \rightarrow G$ such that $\sigma (v) =h$ and $\sigma (w)w =z$ for all $w\in Q$.
Given such a map $\sigma$, denote by $\Sigma :N_{|Q} \rightarrow N_z$ the submersion  
$n \mapsto \Psi_{(\sigma \pi (n),\pi(n))} (n)$. Then, the restriction of $\lambda$ to 
$N_{|Q}$  is the composition $\lambda_z \circ \Sigma$, which shows that $\lambda$ is smooth.

By construction, properties  (\ref{Ite:SecProj}) and (\ref{Ite:HilMap}) are obvious. 
Moreover, $\lambda$ is Gabrielov regular by the following argument. It suffices to verify that 
$\lambda$ is Gabrielov regular on the restricted bundle $N_{|Q}$, where $Q$ is a neighborhood 
of $v$ as before. Since $\lambda_z$ is Gabrielov regular, and $\Sigma : N_{|Q} \rightarrow N_z$ 
a submersion, the composition $\lambda_z \circ \Sigma$ has to be Gabrielov regular as well. 
But over $N_{|Q}$, $\lambda$ coincides with $\lambda_z \circ \Sigma$, which proves 
(\ref{Ite:GabProp}).

It remains to check (\ref{Ite:FormComp}). So let $n\in N_v^{G_v}$ for some $v\in \calO$. 
Then $\tau : N_v \rightarrow N_v$, $\tilde n \mapsto \tilde n  + n$ is a $G_v$-equivariant
affine isomorphism.
Consider a formal power series $F\in \sfJ^\infty (\{n\})^\textup{inv} $, where $n\in N_v^{G_v}$. 
Let $i: N_v \rightarrow N$ be the canonical embedding. Note that $i$ is also $G_v$-equivariant. 
Hence $\tilde F := \tau^{\sharp}_{\{0\},\{n\}} i^{\sharp}_{\{n\},\{n\}} F $ is a $G_v$-invariant formal power 
series at the origin of $N_v$. Thus, by the Theorem of Schwarz and Mather \ref{SchwaSFIACLG}, there exists a
smooth $\tilde f\in \calC^\infty (\R^k)$ such that 
$\tilde F = \sfJ^\infty_{\{ 0\}} \big( \tilde f\circ \lambda_v\big)$. 
Let us show that
$F = \sfJ^\infty_{\{ n \}} \big( f \circ \lambda \big) $, where 
$f (u ) = \tilde f ( u -\lambda_v (n))$, $u\in \R^k$.  Since the submanifolds $N_v \subset N$ 
and $G.\{ n \} \subset N$ are transversal at $n$, it suffices to show that 
\begin{eqnarray}
  \label{eq:first}
   i^{\sharp}_{\{n\},\{n\}} F  & = & i^{\sharp}_{\{n\},\{n\}} \sfJ^\infty_{\{ n \}} \big( f \circ \lambda \big) \: 
   \text{ and } \\
   \label{eq:second}
   \xi_N  \Big( \sfJ^\infty_{\{ n \}} \big( f \circ \lambda \big)\Big) & = & 0 \:
  \text{ for all $\xi \in \frg$}. 
\end{eqnarray}
To verify these equations observe  that $\lambda_v^* f = (\lambda_v^*\tilde f )\circ \tau^{-1}$, hence 
\[
\begin{split}
   i^{\sharp}_{\{n\},\{n\}} F & = \big(\tau^{-1}\big)^{\sharp}_{\{ n \},\{0\}} \tilde F = 
   \big(\tau^{-1}\big)^{\sharp}_{\{ n \},\{0\}} \sfJ^\infty_{\{ 0\}} \big( \tilde f\circ \lambda_v\big) = \\
    & = \sfJ^\infty_{\{ n \}} \big( \tilde f\circ \lambda_v\circ \tau^{-1}\big) =
   \sfJ^\infty_{\{ n \}} \big( f \circ \lambda_v \big) =
   i^{\sharp}_{\{n\},\{n\}} \sfJ^\infty_{\{ n \}} \big( f \circ \lambda \big) . 
\end{split}
\]
This proves Eq.~(\ref{eq:first}). 
Since by construction of $\lambda$, the composition $f\circ \lambda$ is constant on each orbit 
of $G$, Eq.~(\ref{eq:second}) follows, too. 
So we now get 
\[ 
 F = \sfJ^\infty_{\{ n \}} \big( f \circ \lambda \big) 
   =  \lambda^{\sharp}_{\{n\},\{\lambda(n)\}} \sfJ^\infty_{\{ \lambda(n) \}} ( f )
\]
which proves one direction of (\ref{Ite:FormComp}). The other direction is immediate. 

{Step 3.} Let the stratum $S \subset Z$, the point $z \in S$ and the tubular neighborhood $T_\calO$ 
around the orbit $\calO$ through $z$ as in Step 2. Put $O:=T_\calO$. 
In this step, we want to show that
\[
 \iota_{S\cap O}: \calE^\infty(G.S\cap O)^G\to \calE^\infty(S\cap O)^\textup{inv} 
\] 
is onto. To this end, we first transform the claim to an equivalent statement about 
Whitney functions on certain subsets of the normal bundle $N$. 
Observe that $\calO\subset V$ is a Nash submanifold by \cite[A.~10]{BryIKGNO}
and that $O = T_\calO$ is a Nash tubular neighborhood 
of $\calO$ by \cite[Cor.~8.9.5]{BocCosRoyRAG}, so in particular semialgebraic. 
Morever, the restriction 
$\exp_{|\widetilde O}$ of the exponential map to $\widetilde O = \exp^{-1} (T_\calO) \cap N$ 
is a Nash diffeomorphism from $O$ to  $\widetilde O$. 
This implies in particular that $\widetilde S = \exp^{-1} (S \cap O) $ is a closed
analytic submanifold of $\widetilde O$. Moreover, it entails that the claimed surjectivity of 
$\iota_{S\cap O}$ is equivalent to 
\[
  \iota_{\widetilde S} := \big( \id^{\sharp}_{\widetilde S, G.\widetilde{S}}\big)_{|\calE^\infty(G.\widetilde{S})^G} : 
  \calE^\infty(G.\widetilde{S})^G\to \calE^\infty(\widetilde S)^\textup{inv} 
\]
being onto. So let us prove this. Choose an open subset $\Omega \subset \R^k$ such that 
$\lambda (\widetilde S)$ is closed in $\Omega$; this is possible by Prop.~\ref{Prop:TopProp} (\ref{It:ExOpNbhd}). 
The restriction $\lambda_{|\widetilde O} :\widetilde O \rightarrow \Omega$ now is a proper analytic map.  
Next consider the morphism 
\begin{equation}
\label{eq:lambdasharp}
  \lambda_{\widetilde S, \lambda(\widetilde S)}^{\sharp} : \calE^\infty ( \lambda(\widetilde S) ) \rightarrow 
  \calE^\infty (\widetilde S).
\end{equation}
We claim that the image of this map is $\calE^\infty (\widetilde S)^\textup{inv}$.
Since $\lambda_{|\widetilde O}$ is proper and Gabrielov regular at each point of $\widetilde O$ by Step 2, 
we can apply Theorem 3.6  by Bierstone--Milman \cite{BieMilCDF}. Thus, the following equality holds true:
\begin{equation}
  \label{Eq:ImEquiv}
  \lambda_{\widetilde S, \lambda(\widetilde S)}^{\sharp} \calE^\infty (\lambda(\widetilde S)) =
  \overline{\lambda_{\widetilde S, \lambda(\widetilde S)}^{\sharp} \calE^\infty (\lambda(\widetilde S))}=
  \big( \lambda_{\widetilde S, \lambda(\widetilde S)}^{\sharp} \calE^\infty (\lambda(\widetilde S))\big)^\wedge ,
\end{equation}
where the last term denotes the algebra of Whitney functions on $\widetilde S$ which
are a formal composite with $\lambda$. 
By  (\ref{Ite:FormComp}) in Step 2 one concludes that $F \in  \calE^\infty (\widetilde S)$ is  a formal composition with $\lambda$
if and only if $F$ is an element of $\calE^\infty (\widetilde S)^\textup{inv}$. Hence the 
image of the map  (\ref{eq:lambdasharp}) is $\calE^\infty (\widetilde S)^\textup{inv}$ indeed. 
Now consider the commutative diagram:  
 \begin{eqnarray*}
\xymatrix{
 \calE^\infty (\lambda (\widetilde S )) 
 \ar[rr]^{\mbox{ }\quad \lambda_{G.\widetilde{S},\lambda (G. \widetilde{S} )}^{\sharp}}
 \ar[drr]_{\mbox{ }\quad \lambda_{\widetilde{S},\lambda ( \widetilde{S} )}^{\sharp}}&  &
 \calE^\infty(G.\widetilde{S})^G \ar[d]^{\iota^{\sharp}_{\widetilde{S}}}
 \\ & & \calE^\infty(\widetilde{S})^\textup{inv}
}
\end{eqnarray*} 
Since by the above 
$\lambda_{\widetilde{S},\lambda ( \widetilde{S} )}^{\sharp} (\calE^\infty (\lambda (\widetilde{S})))=
\calE^\infty(\widetilde{S})^\textup{inv}$, it follows that the vertical arrow in the 
diagram is surjective. This finishes Step 3.

{Step 4.} Again let $S \subset Z$ be  a stratum of highest depth.  
Assume that $ E \in \calE^\infty (Z)^\textup{inv}$ is an invariant Whitney function
such that $E (z) =0$ for all $z\in  S$ and that  
$\sfE_{ Z \setminus S} (\operatorname{res}^{Z}_{Z \setminus S} E) \in 
  \calE^\infty (G. Z \setminus G. S )^G$. 
By the proof of Prop.~\ref{idiso}  one knows that  $\sfE_{Z}(E) \in \sfJ^\infty (G.Z)$. 
Moreover, 
\[ \operatorname{res}^{G.Z}_{G.S} \sfE_{Z}(E) = \sfE_{S} (\operatorname{res}^{Z}_{S} E) = 0 \]
and 
\[ \operatorname{res}^{G.Z}_{G.Z \setminus G.S} \sfE_{Z}(E) = \sfE_{Z \setminus S} (\operatorname{res}^{Z}_{Z\setminus S} E)  
   \in \calE^\infty (G. Z \setminus G. S )^G \ . \] 
By Hestenes's lemma for subanalytic sets \ref{Lem:subanalytic-hestenes-lemma} one obtains
$\sfE_{Z}(E) \in  \calE^\infty (G. Z )^G$ and Claim B is proved.
\end{proof}

\begin{proof}[Proof of Theorem \ref{maindiag}] We already know that all maps in the diagram of Theorem \ref{maindiag} are well-defined and that the diagram is in fact commutative. The middle column  is exact by the Theorem of Schwarz (cf.~Theorem \ref{SchwaSFIACLG}), the first column is exact by Proposition \ref{flatalongZ}, the middle row is exact by the Whitney Extension Theorem (cf.~Theorem \ref{WhitneyExtension}) and the first row is exact by Corollary \ref{surgrpd}. The third row is clearly exact. 

So it remains to show that the third column is exact. This can be done by an elementary diagram chase as follows. We have to show that \[\Kern\left(\rho^{\sharp}_{Z,\rho(Z)}\right)=\mathsf J^\infty_{\rho(Z)}(\mathcal J_X).\]  
To this end let us assume that $\rho^{\sharp}_{Z,\rho(Z)}(F)=0$. By Whitney's Extension Theorem there exists a function $f\in \mathcal C^\infty(\mathbb R^\ell)$ such that $\mathsf J^\infty_{\rho(Z)}(f)=F$. Moreover, since the upper right square in our diagram commmutes, $\rho^*(f)$ has to be flat along $Z$. Since the upper left arrow $\rho^*$ is onto, there exists a function $g\in \calC^\infty(\mathbb R^\ell)$ that is flat along $\rho(Z)$ such that $\rho^*(f-g)=0$. Clearly $\mathsf J^\infty_{\rho(Z)}(f-g)=F$ and $f-g\in \calJ_X$, which proves our claim.
\end{proof}

\begin{corollary} \label{corollary} 
Under the assumptions of Theorem \ref{maindiag}, the continuous linear surjections 
\begin{eqnarray*}
\rho^\sharp_{Z,\rho(Z)}&:&\calE^\infty(\rho(Z))\to \calE^\infty(Z)^\textup{inv}\\
\rho^\sharp_{G.Z,\rho(Z)}&:&\calE^\infty(\rho(Z))\to \calE^\infty(G.Z)^G
\end{eqnarray*}
induce isomorphisms of Fr{\'e}chet algebras
\begin{eqnarray*}
\xymatrix{&\calE^\infty(\rho(Z),X)\ar[dl] \ar[dr] &\\
\calE^\infty(G.Z)^G\ar[rr]_{\id^\sharp_{Z,G.Z}}&&\calE^\infty(Z)^\textup{inv}.}
\end{eqnarray*}
\end{corollary}
\begin{proof}[Proof of Corollary \ref{corollary}]
Observe that the pullback $\id^\sharp_{Z,G.Z}$ of the identity $\id:V\to V$ associates to a Whitney function $F=(F_\alpha)_{\alpha\in \mathbb N^n}$ on $G.Z$ the Whitney function  $F_{|Z}:=({F_\alpha}_{|Z})_{\alpha\in \mathbb N^n}$ obtained by restriction to $Z$. Hence $\id^\sharp_{Z,G.Z}$ is injective. We also have  to show that the restriction of $\id^\sharp_{Z,G.Z}$ to $\calE^\infty(G.Z)^G$ is onto $\calE^\infty(Z)^\textup{inv}$. But in view of  
\begin{eqnarray}\label{idrhorho}
\id^\sharp_{Z,G.Z}\circ\rho^\sharp_{G.Z,\rho(Z)} =\rho^\sharp_{Z,\rho(Z)}\:,
\end{eqnarray}
$\id^\sharp_{Z,G.Z}$ has to be onto since $\rho^\sharp_{Z,\rho(Z)}$ is. So we have proven that the horizontal arrow in the diagram is an isomorphsm. That the other two arrows are isomorphisms follows from  $\Kern\left(\rho^\sharp_{Z,\rho(Z)}\right)=\sfJ^\infty_{\rho(Z)}(\calJ_X)=\Kern\left(\rho^\sharp_{G.Z,\rho(Z)}\right)$.
\end{proof}

\begin{appendix}
\section{Tools  from differential analysis}
For the convenience of the reader we collect in this appendix some tools from 
Differential Analysis which we need in this paper and which appear to be 
scattered through the literature. 

\subsection{The formula of Fa{\`a} di Bruno}
\label{sec:formula-faa-di-bruno}
Here we recall the combinatorial and the multiindex version of the formula 
of Fa{\`a} di Bruno. 

Thoughout this section $\varphi=(\varphi^1, \dots ,\varphi^m): U\to V$ denotes a smooth map between open subsets $U\subset \mathbb R^n$ and $V\subset \mathbb R^m$ and $f\in \calC^\infty(V)$. Coordinates for $U$ and $V$  are denoted by $x^1,\dots,x^n$ and $y^1,\dots,y^m$, respectively. We are interested in the higher order partial derivatives of $f\circ \varphi$. 

By a \emph{partition into $k\ge 1$ blocks} of the set $[\ell]:=\{1,\dots,\ell\}$ we mean a decomposition into nonempty subsets $I_1\sqcup\dots\sqcup I_k=[\ell]$, called blocks, regardless of the order of the blocks. The number of partitions of  $[\ell]$ into $k$ blocks is given by the Stirling number of the second kind $\{\tiny\begin{matrix}\ell\\k\end{matrix}\tiny\}$. In total the number of partitions of $[\ell]$ is given by the Bell numbers $B_\ell$; the first few Bell numbers are $1, 2, 5, 15, 52, 203,\dots$.

Let us choose some indices $i_1,\dots,i_\ell \in \{1,\dots m\}$. Given a subset $I=\{k_1,\dots,k_r\}\subset [\ell]$ of cardinality $r=|I|$ we introduce the shorthand notation 
\[\frac{\partial^{|I|}}{\partial x^I}:=\frac{\partial^r}{\partial x^{i_{k_1}}\dots \partial x^{i_{k_r}}}\:.\]

\begin{theorem}[Formula of Fa{\`a} di Bruno -- combinatorial version]\label{FdBcomb} With the notation from above we have for all $x\in U$
\begin{eqnarray*}
\lefteqn{\left(\frac{\partial^\ell}{\partial x^{i_1}\dots \partial x^{i_\ell}}f\circ \varphi\right)(x)}&&\\
\nonumber&=&\sum_{k=1}^\ell\quad\sum_{I_1\sqcup\dots\sqcup I_k=[\ell]}\quad\sum_{j_1,\dots,j_k=1}^m\frac{\partial^kf}{\partial y^{j_1}\dots \partial y^{j_k}}(\varphi(x))\:\frac{\partial^{|I_1|} \varphi^{j_1}}{\partial x^{I_1}}(x)\dots\frac{\partial^{|I_k|} \varphi ^{j_k}}{\partial x^{I_k}}(x),
\end{eqnarray*}
where the summation is taken over decompositions of $[\ell]=\{1,\dots,\ell\}$ into blocks as described above.\end{theorem}

The combinatorial version of the Fa{\`a} di Bruno formula is easy to prove, it contains no combinatorial factors and it is clear how it transforms under linear change of coordinates. Using multiindex notation, symmetry factors have to be taken into account and things look slightly more complicated. Let us consider maps of the form
\[\lambda:[m]\times (\mathbb N^n\backslash\{0\})\to \mathbb N,\qquad(i,\alpha)\mapsto \lambda_{i,\alpha}\] 
and let us denote the set of such maps by $\Lambda_{n,m}$. Picking a multiindex $\beta\in\mathbb N^n$, we observe that there are only finitely many  $\lambda\in\Lambda_{n,m}$ that solve the constraint $\sum_{i,\alpha}\lambda_{i,\alpha}\alpha=\beta$. Let us denote the set of such solutions by  $\Lambda_{n,m}(\beta)$. For each  $\lambda \in \Lambda_{n,m}(\beta)$ only finitely many $\lambda_{i,\alpha}$ are nonzero, and hence
  \[\lambda!:=\prod_{i,\alpha}\lambda_{i,\alpha}!\] is well-defined. For fixed $\alpha$, $\lambda_{\alpha}=(\lambda_{1,\alpha},\dots,\lambda_{m,\alpha})$ is viewed as a multiindex in $\mathbb N^m$. Summing these terms up we end up with $\sum_\alpha \lambda _\alpha\in\mathbb N^m$.

\begin{theorem}[Fa\`a di Bruno -- multiindex version, cf.~\cite{MichTG}]\label{FdBmulti} With the notation from above we have
  \begin{align*}
    \partial^\beta (f\circ \varphi) = \sum_{\lambda \in \Lambda_{n,m}(\beta)} \frac{\beta !}{\lambda !} \quad \left(\partial ^{\sum_{\alpha}\lambda_\alpha} f\right)\circ \varphi \quad \prod_{\alpha}\frac{\left(\partial^\alpha \varphi\right)^{\lambda_\alpha}}{(\alpha!)^{\sum_i \lambda_{i,\alpha}}}
  \end{align*}
\end{theorem}
\begin{proof} The proof can be achieved by combining the Taylor expansion with the following formula, which is a corollary of the multinomial theorem. For $b\in \mathbb N$ and a formal series $\sum_{\alpha\in \mathbb N^n} a_\alpha \bs x^\alpha\in \mathbb R [\![\bs x]\!]= \mathbb R[\![x_1,\dots,x_n]\!]$ we have
\[\left(\sum_{\alpha\in \mathbb N^n} a_\alpha \bs x^\alpha\right)^b=\sum_{\nu\in \mathbb N^{\mathbb N^n}\colon \sum_{\alpha\in \mathbb N^n}\nu_\alpha=b}b!\left(\prod_\alpha \frac{a_\alpha^{\nu_\alpha}}{\nu_\alpha !}\right)\: \bs x^{\sum_{\alpha}\nu_\alpha \alpha} \: .\]   
\end{proof}

\subsection{Gabrielov regularity}
\label{App:Gabrielov}

Let $\varphi: X\to Y$ be a real analytic map between real analytic manifolds
$X$ and $Y$ and let $x\in X$. Denote by $\mathcal O_{X,x}$ the algebra of 
germs of real analytic functions on $X$  at $x$ and by
$ \sfJ^\infty_{X} (\{x\})$ the algebra of formal power series 
or in other words the algebra of infinite jets on $X$ at $x$. 
Using the notation from Section \ref{sec:whitney-functions},
$\sfJ^\infty (\{x\}) = \mathcal{E}^\infty (\{ x \},X) = \calC^\infty (X)/ \calJ^\infty (\{ x \} , X) $.
Similarly, we use $\mathcal O_{Y,y}$ and 
$\sfJ^\infty (\{ y \} )$ for the corresponding algebras at 
$y\in Y$. The map $\varphi$ induces algebra morphisms
$\varphi_x^*: \mathcal O_{Y,\varphi(x)}\to \mathcal O_{X,x}$ and
$\varphi_{x,\varphi(x)}^{\sharp}:  \sfJ^\infty (\{ \varphi(x) \} )\to  \sfJ^\infty (\{x\})$,
where the latter is  given in local coordinates by Taylor expansion as in Eq.~\eqref{Eq:DefPullbackJet}.
Gabrielov \cite{GabFRAF} uses the following three notions of rank of $\varphi$ at $x$:
\begin{itemize}
\item $r_1(x)=$ rank of the Jacobian of $\varphi$ at a generic point nearby $x$,
\item $r_2(x)= \dim \big( \sfJ^\infty (\{ \varphi(x) \} ) / \ker \varphi_{x,\varphi(x)}^{\sharp} \big)$,
\item $r_3(x)=\dim \big( \mathcal O_{Y,\varphi(x)}/\ker \varphi_x^* \big)$.
\end{itemize}
One observes $r_1(x)\le r_2(x)\le r_3(x)$. The map $\varphi$ is said to be 
\emph{regular in the sense of Gabrielov at} $x$ if $r_1(x)=r_3(x)$, i.e., if the three ranks 
coincide. It is said to be \emph{regular in the sense of Gabrielov} if it has this property 
at all $x\in X$.

\begin{proposition}
\label{Prop:GabrielovRegularityPolynomials}
  Let  $\varphi: X= \R^n \to Y =\R^m$ be a polynomial mapping. Then $\varphi$ is Gabrielov regular. 
\end{proposition}

\begin{proof}
  The proof follows Bierstone \cite{BiePC}.
  By the Tarski-Seidenberg Theorem, the image of $\varphi$ is a semialgebraic subset of $Y=\R^m$.
  By  a result of {\L}ojasewicz \cite{LojESA}, \cite[Thm.~2.13]{BieMilSSS}, $\operatorname{im} \varphi$ is contained 
  in an algebraic set of the same dimension. The ranks of Gabrielov therefore coincide at every point 
  of the domain. 
\end{proof}

A class of subanalytic sets having the so-called ''composite function property'' \cite{BieMilCDF,BieMilPawCDF} 
is the one defined in the following. 

\begin{definition}[{\cite{BieMilRAFI,BieMilCDF}}] 
\label{Def:NashSubanalytic}
  A subset $X\subset N$ of a real analytic manifold $N$ is called \emph{Nash subanalytic} if $X$ is 
  the image of a proper Gabrielov regular real analytic map $\varphi : M \to N$ defined on a 
  real analytic manifold $M$.  
\end{definition}

\subsection{A generalized Hestenes lemma}
\label{sec:subanalytic-hestenes-lemma}

Last in this appendix we recall a subanalytic version of the lemma by Hestenes by 
Bierstone--Milman \cite{BieMilCDF,BieMilGDPSS} which gives a criterion when a
jet on a subanalytic set is in fact a Whitney function. 

\begin{lemma}[Hestenes's lemma for subanalytic sets, {\cite[Cor.~8.2]{BieMilCDF} \& \cite[Lem.~11.4]{BieMilGDPSS}}]
\label{Lem:subanalytic-hestenes-lemma}
Let $U \subset \R^n$ be open and $A,B  \subset U$ two relatively closed subanalytic subsets such that $B\subset A$. 
Let $G \in \sfJ^\infty (A)$ and assume that $\operatorname{res}^A_{B} G \in \calE^\infty (B)$ and
$\operatorname{res}^A_{A\setminus B} G \in \calE^\infty (A\setminus B)$. Then $G$ is a Whitney function
on $A$ that is $G \in \calE^\infty (A)$.
\end{lemma}

\end{appendix}
\bibliography{HerPflHHASFOS}
\end{document}